\documentclass[a4paper]{amsart}

\usepackage[utf8]{inputenc}
\usepackage[british]{babel}
\usepackage{amsfonts,amsmath,amssymb,amsthm,mathtools,esint}
\usepackage[dvipsnames]{xcolor}
\usepackage{tikz}
\usepackage{pgfplots}
\pgfplotsset{compat=1.3}
\usepackage{caption}
\usepackage[justification=centering]{caption}

\usepackage[hidelinks]{hyperref}
\usepackage{cleveref}
\crefname{subsection}{subsection}{subsections}
\numberwithin{equation}{section}

\newcommand{\dist}{\mathrm{dist}}

\newcommand{\R}{\mathbb{R}}
\renewcommand{\d}[1]{\,\mathrm{d}#1}
\renewcommand{\div}{\mathrm{div}}
\newcommand{\I}{\mathrm{I}}
\newcommand{\tr}{\mathrm{tr}}
\newcommand{\M}{\mathbb{M}}
\newcommand{\D}{\mathrm{D}}
\newcommand{\rot}{\mathrm{rot}}

\newtheorem{theorem}{Theorem}[section]
\newtheorem{lemma}[theorem]{Lemma}
\newtheorem{corollary}[theorem]{Corollary}
\theoremstyle{remark}
\newtheorem{remark}[theorem]{Remark}
\theoremstyle{definition}
\newtheorem{definition}[theorem]{Definition}
\newtheorem{assumption}{Assumption}

\begin{document}

\title[Regularization of the Monge--Amp\`ere equation]{
        Convergence of a regularized finite element discretization 
         of the two-dimensional Monge--Amp\`ere equation}
\author[D.~Gallistl, N.~T.~Tran]
      {Dietmar Gallistl \and Ngoc Tien Tran}
\thanks{This project received funding from
        the European Union's Horizon 2020 research and innovation 
        programme (project DAFNE, grant agreement No.~891734).
       The authors thank Lukas Gehring for stimulating 
       discussions and for proposing
       the fourth example in the numerical experiments
       as well as two anonymous referees for their valuable comments
       which led to improvements of the manuscript.}
\address[D.~Gallistl, N.~T.~Tran]{%
         Institut f\"ur Mathematik,
         Friedrich-Schiller-Uni\-ver\-si\-t\"at Jena,
         07743 Jena, Germany}
         \email{ $\{$dietmar.gallistl, ngoc.tien.tran$\}$@uni-jena.de}
\date{\today}


\begin{abstract}
This paper proposes a regularization 
of the Monge--Amp\`ere equation in planar convex domains
through uniformly elliptic
Hamilton--Jacobi--Bellman equations.
The regularized problem possesses a unique strong solution
$u_\varepsilon$
and is accessible to the discretization with finite elements.
This work establishes uniform convergence of 
$u_\varepsilon$ to the convex 
Alexandrov solution $u$ to the
Monge--Amp\`ere equation as the regularization parameter
$\varepsilon$ approaches $0$.
A mixed finite element method for the approximation of $u_\varepsilon$
is proposed, and the regularized finite element scheme is
shown to be uniformly convergent.
The class of admissible right-hand sides 
are the functions that can be approximated from
below by positive continuous functions in the $L^1$ norm.
Numerical experiments provide empirical evidence 
for the efficient approximation of singular solutions $u$.
\end{abstract}
	
\maketitle

\section{Introduction}\label{sec:introduction}

This work studies a regularized formulation of the 
two-dimensional Monge--Amp\`ere equation, which makes
the problem accessible to discretizations with the 
finite element method (FEM).

\subsection{Background and motivation}
Given a bounded open convex domain $\Omega \subset \R^2$
in the plane, a nonnegative right-hand side
$f: \Omega \to \R_{\geq 0}$,
and boundary data $g \in C(\overline{\Omega})$,
the Monge--Amp\`ere equation seeks a convex solution $u$ to
\begin{align}\label{pr:Monge-Ampere}
	\det \D^2 u = f \text{ in } \Omega 
	\quad\text{and}\quad u = g \text{ on } \partial \Omega.
\end{align}
If the boundary data is inhomogeneous, i.e.\ $g\neq 0$
on $\partial\Omega$, we additionally assume that
$\Omega$ is strictly convex.
This restriction is required for the existence and strict convexity of weak solutions to \eqref{pr:Monge-Ampere}
and is rather related to the PDE than to the algorithm of this paper.
The Monge--Amp\`ere equation \eqref{pr:Monge-Ampere} arises from
applications in, e.g., optimal transport or differential geometry (prescribed Gaussian curvature). 
It is a fully nonlinear degenerate elliptic 
partial differential equation (PDE); in particular it is
in nondivergence form.
Hence, concepts of generalized solutions are based on maximum and
comparison principles or geometric considerations rather than on weak
differential operators arising from an integration by parts
(see \Cref{ss:generalizedsol} below for definitions).
This makes the numerical approximation of the problem
generally difficult.

In the context of fully nonlinear PDE, finite difference methods (FDM) 
are well studied with general convergence results 
\cite{BarlesSouganidis1991}.
For equations of the Monge--Amp\`ere type,
wide stencil FDMs were proposed and analyzed in, e.g.,
\cite{FroeseOberman2011,FroeseOberman2013,
BenamouCollinoMirebeau2016,Mirebeau2016,NochettoNtogkasZhang2019}. 
These schemes rely on monotonicity of the method
in the sense that certain maximum principles are
respected on the discrete level. 
However, natural limitations arise in that they are restricted to
low-order methods on structured meshes, and fixed size finite
difference stencils are generally not sufficient for a convergent 
scheme \cite{MotzkinWasow1953}.
The design of finite elements for the Monge--Amp\`ere
equation started with the geometric
scheme \cite{OlikerPrussner1988}.
Other existing finite element methods
involve consistent linearization of \eqref{pr:Monge-Ampere}, cf., 
e.g., \cite{BrennerGudiNeilanSung2011,Neilan2014,Awanou2017},
or the addition of 
small perturbations via the bi-Laplacian 
\cite{FengNeilan2011,FengNeilan2014}. 
The survey article 
\cite{NEILAN2020105} provides a thorough overview on FDMs and FEMs for 
the numerical approximation of solutions to \eqref{pr:Monge-Ampere}.
We furthermore refer to the review articles
\cite{FengGlowinskiNeilan2013,NeilanSalgadoZhang2017}
on more general fully nonlinear PDEs.

One major difficulty in the approximation of the Monge--Amp\`ere
equation is caused by the constraint on the solution to be convex.
It is well known \cite{Gutierrez2016} that the solution to
\eqref{pr:Monge-Ampere} is unique only in the cone of convex
functions and the problem is only elliptic on the set of
positive semidefinite matrices.
Yet, the problem can be reformulated as a
Hamilton--Jacobi--Bellman equation
(without convexity constraints) as pointed out by
\cite{Krylov1987}.
This fact was used by \cite{FengJensen2017} as a basis for 
a numerical method.
The formulation is as follows.
For $\M \coloneqq \R^{2 \times 2}$, let $\mathbb{S} \subset \M$ denote the
subspace of 
symmetric matrices, 
$\mathbb{S}_+$ the set of positive semidefinite symmetric matrices,
and $\mathbb{S}(0) \coloneqq \{A \in \mathbb{S}_+: \tr A = 1\}$
the subset of matrices with normalized trace. 
The equivalent HJB equation reads, for continuous right-hand side $f \in C(\Omega)$,
\begin{align}\label{pr:Monge-Ampere-HJB}
	F(f; x, \D^2 u(x)) = 0 \text{ in } \Omega 
	\quad\text{and}\quad  u= g \text{ on } \partial \Omega
\end{align}
for $F(f; \cdot, \cdot) \in C(\Omega \times \mathbb{S})$ 
given by
\begin{align}
	F(f; x, M) \coloneqq 
	\sup_{A \in \mathbb{S}(0)} (- A:M + 2\sqrt{f(x)\det A})
	\label{def:F}
\end{align}
for all $x \in \Omega$ and $M \in \mathbb{S}$.
Major progress on the finite element discretization of
HJB equations was initiated by
\cite{SmearsSueli2013,SmearsSueli2014}.
Those papers establish existence and uniqueness 
as well as finite element convergence theory
of the strong solutions $u \in H^2(\Omega)$
provided the Cordes condition 
\cite{MaugeriPalagachevSoftova2000} holds.
It is a pointwise 
dimension-dependent algebraic condition on the ratio
of the Frobenius norm and the trace of the coefficient,
which in the two-dimensional case (relevant to this paper)
states that there is some $\delta\in(0,1]$ such that
\begin{equation}\label{e:cordes}
	|A|^2/(\tr A)^2 \leq 1/(1 + \delta)
	\quad\text{uniformly in }A,
\end{equation}  
and is equivalent to the uniform ellipticity of $F(f;\cdot,\cdot)$.
This is \emph{not} the case in higher space dimensions.
Problem \eqref{pr:Monge-Ampere-HJB} is however only \emph{degenerate}
elliptic because the eigenvalues of $A$ in \eqref{def:F} may become arbitrary small
and \eqref{e:cordes} may fail to hold.
The point of departure of the regularization approach proposed
in this paper is to replace $F$ by a uniformly elliptic operator
$F_\varepsilon$ by replacing $\mathbb S(0)$ by a 
compact subset allowing for uniformly positive eigenvalues only.

\subsection{Contributions of this paper}
The proposed regularized PDE is uniformly elliptic for each
regularization parameter $0 < \varepsilon \leq 1/2$
and, in particular, satisfies the Cordes condition.
For convex domains $\Omega$, $f \in L^1(\Omega)$ with $f \geq 0$, and $g \in H^2(\Omega)$, 
the framework of \cite{SmearsSueli2013,SmearsSueli2014} provides a
unique strong solution $u_\varepsilon \in H^2(\Omega)$.
The main results in \Cref{sec:convergence} establish the
uniform convergence of $u_\varepsilon$ to $u$ in $\Omega$ as
$\varepsilon \searrow 0$ 
for right-hand sides $f$ from the class
$L^1_\uparrow(\Omega)$ defined in \Cref{def:L1plus} below
and sufficiently smooth boundary data $g \in H^2(\Omega) \cap C^{1,\beta}(\overline{\Omega})$
with $0 < \beta < 1$. 
Here, $u \in C(\overline{\Omega})$ is the Alexandrov solution to the 
Monge--Amp\`ere equation \eqref{pr:Monge-Ampere} defined in \Cref{ss:generalizedsol} below.
If $f \in C^{2,\alpha}(\Omega)$ 
with $0 < \alpha < 1$, $f > 0$ in $\overline{\Omega}$, 
and $g \equiv 0$, then the convergence rate
\begin{align*}
	\|u - u_\varepsilon\|_{L^\infty(\Omega)}
	\lesssim \varepsilon^{1/32}
\end{align*}
is established, which is of theoretical interest
(but possibly of limited practical impact
due to the very small exponent).
The solution $u_\varepsilon$ to the regularized problem
can be approximated with the FEM.
This paper utilizes the mixed finite element method in 
\cite{GallistlSueli2019} for the approximation of $u_\varepsilon$.
Although the particular choice of the finite element discretization
to the regularized problem is not essential,
we make this particular choice because it only requires
standard Lagrange finite element spaces and thus allows
for a rather transparent error analysis.
The application of other well-established methods for HJB equations
under the Cordes condition, e.g.~\cite{SmearsSueli2014,BrennerKawecki2021}, would also be possible.
An advantage of the regularized FEM approach is that
existence and uniqueness of meaningful discrete solutions
do not rely on assumptions on the mesh to be sufficiently fine.
The error estimates proven in this work are valid on arbitrary
coarse meshes, but good approximations may ---at least in theory---
require small mesh-sizes depending on the regularization.
The limitation of the approach is that
the results are restricted to two space dimensions due to the severe 
restrictions imposed by the Cordes condition in higher dimensions.

\subsection{Notation}
Standard notation for function spaces applies throughout this paper. Let $C^k(\Omega;\R^d)$ for $k \in \mathbb{N}$ denote the space of $k$-times continuously differentiable functions with values in $\R^d$. Given a positive parameter $0 < \alpha \leq 1$, the H\"older space $C^{k,\alpha}(\Omega;\R^d)$, endowed with the norm $\|\cdot\|_{C^{k,\alpha}(\Omega)}$, is the subspace of $C^k(\overline{\Omega};\R^d)$ such that all partial derivates of order $k$ are H\"older continuous with exponent $\alpha$.
The local H\"older space $C^{k,\alpha}_{\text{loc}}(\Omega;\R^d)$ denotes the set of $v \in C^k(\Omega;\R^d)$ with $\|v\|_{C^{k,\alpha}(\Omega')} < \infty$ for all compact subsets $\Omega' \Subset \Omega$ of $\Omega$.
For $d = 1$, abbreviate $C^k(\Omega) \coloneqq C^k(\Omega;\R)$, $C^{k,\alpha}(\Omega) \coloneqq C^{k,\alpha}(\Omega;\R)$, and $C^{k,\alpha}_{\text{loc}}(\Omega) \coloneqq C^{k,\alpha}_{\text{loc}}(\Omega;\R)$.
This convention also applies to Lebesgue and Sobolev functions $L^2(\Omega;\R^d)$, $H^1(\Omega;\R^d)$, $H^2(\Omega;\R^d)$ etc.
The space of Sobolev functions in $H^1(\Omega;\R^d)$ with vanishing tangential trace is denoted by $H^1_t(\Omega;\R^d)$. Throughout this paper, let $(\cdot,\cdot)_{L^2(\Omega)}$ denote the $L^2$ scalar product and abbreviate $V \coloneqq H^1_0(\Omega)$, $W \coloneqq H^1_t(\Omega;\R^2)$.

A sequence $(v_j)_{j \in \mathbb{N}}$ of continuous functions $v_j \in C(\Omega)$ converges locally uniformly to $v \in C(\Omega)$ if $\lim_{j \to \infty} \|v_j - v\|_{L^\infty(\Omega')} = 0$ holds for all compact subset $\Omega' \Subset \Omega$ of $\Omega$. If $\lim_{j \to \infty} \|v_j - v\|_{L^\infty(\Omega)} = 0$, then $(v_j)_{j \in \mathbb{N}}$ converges uniformly to $v$ in $\Omega$.
	
For $A,B \in \M$, the Euclidean scalar product $A:B \coloneqq \sum_{j,k = 1}^{2} A_{jk} B_{jk}$ induces the Frobenius norm $|A| \coloneqq \sqrt{A:A}$ in $\M$. 
The notation $|\cdot|$ also denotes the absolute value of a scalar
or the length of a vector.
The relation $A \leq B$ of $A,B \in \mathbb{S}$ holds whenever $B - A \in \mathbb{S}_+$ is a non-negative definite symmetric matrix.
The notation $a \lesssim b$ abbreviates $a \leq C b$ for a generic constant $C$ independent of the mesh-size and regularization parameter $\varepsilon$, and $a \approx b$ abbreviates $a \lesssim b \lesssim a$.
	
\subsection{Outline}
The remaining parts of this paper are organized as follows.
\Cref{sec:review-Monge-Ampere} reviews different concepts of weak solutions 
to the Monge--Amp\`ere equation \eqref{pr:Monge-Ampere} and their properties.
\Cref{sec:regularization} introduces
and analyzes the regularized PDE,
while \Cref{sec:convergence} establishes the convergence towards the solution to
the Monge--Amp\`ere equation \eqref{pr:Monge-Ampere}. \Cref{sec:discretization} presents the mixed finite element method from \cite{GallistlSueli2019} for the approximation of the solution to the regularized PDE. The numerical experiments in \Cref{sec:numerical-experiments} conclude this paper.

\section{Review on the Monge--Amp\`ere equation}\label{sec:review-Monge-Ampere}
This section recalls two different concepts of weak solutions to the Monge--Amp\`ere equation \eqref{pr:Monge-Ampere} in the literature.
Further details can be found
in the monographs \cite{Gutierrez2016,Figalli2017}.

\subsection{Generalized solution}\label{ss:generalizedsol}
Throughout the remaining parts of this paper, let $\Omega \subset \R^2$ be an open bounded Lipschitz domain in the plane.
For the existence and uniqueness of generalized solutions to the Monge--Amp\`ere equation \eqref{pr:Monge-Ampere} in \Cref{thm:existence-uniqueness-Monge-Ampere} below,
we will use the following
basic assumptions on the data.
It is not essential for our algorithm, but is required for existence and strict convexity of generalized solutions on the PDE level.
\begin{assumption}\label{assumption:structure}
The domain $\Omega \subset \R^2$ is open, bounded, and convex.
The right-hand side $f\in L^1(\Omega)$ is a nonnegative function
$f \geq 0$ and the boundary data is prescribed by
some $g\in C(\overline{\Omega})$.
If the boundary data is inhomogeneous, that is $g \neq 0$ on $\partial\Omega$,
then $\Omega$ is assumed to be strictly convex.
\end{assumption}
Any nonnegative function $f\in L^1(\Omega)$ induces
the Borel measure $\int_E f\d x$ for any measurable subset $E\subset\Omega$.
The most general weak solution concept to the Monge--Amp\`ere
equation admits Borel measures as right-hand sides.
Given a convex function $u \in C(\Omega)$, the subdifferential $\partial u(x) \subset \R^2$ of $u$ at $x \in \Omega$ is the set
\begin{align*}
	\partial u(x) \coloneqq \{a \in \R^2: u(x) + a\cdot (y-x) \leq u(y) \text{ for all } y \in \Omega\}.
\end{align*}
The Monge--Amp\`ere measure $\mu_u$ associated to $u$ is defined by
\begin{align}
	\mu_u(E) \coloneqq 
	\mathcal L^2(\cup_{x \in E} \partial u(x))
 \text{ for all Borel sets } E \subset \Omega,
	\label{def:Monge-Ampere-measure}
\end{align}
where $\mathcal L^2$ is the two-dimensional 
Lebesgue measure.
It turns out that $\mu_u$ is a Borel measure \cite{Alexandrov1958}
and, if $u \in C^2(\Omega)$, 
$\mu_u(E) = \mathcal L^2(\nabla u(E)) = \int_E \det \D^2 u\d x$.
This motivates the following weak solution
concept \cite{Alexandrov1958}.
\begin{definition}[Alexandrov solution]\label{def:Alexandrov}
	Given a Borel measure $\nu$ in $\Omega$ and $g \in C(\overline{\Omega})$, a convex function $u \in C(\overline{\Omega})$ 
	with the Monge--Amp\`ere measure $\mu_u$ from \eqref{def:Monge-Ampere-measure} is called Alexandrov solution to \eqref{pr:Monge-Ampere} if $\mu_u = \nu$ and $u = g$ on $\partial \Omega$.
\end{definition}

If the right-hand side is represented by a continuous function
$f \in C(\Omega)$, the concept of viscosity solutions
from the theory of fully nonlinear second-order PDEs 
\cite{CaffarelliCabre1995,Gutierrez2016}
applies.

\begin{definition}[convex viscosity solution]\label{def:convex-viscosity}
	Let $f \in C(\Omega)$ with $f \geq 0$. A convex function $u \in C(\Omega)$ is called viscosity subsolution (resp.~supersolution) to
	\begin{align}
		\det \D^2 u = f \text{ in } \Omega
		\label{def:det=f}
	\end{align}
	if, for all $x_0 \in \Omega$ and convex $\varphi \in C^2(\Omega)$ such that $u - \varphi$ has a local maximum (resp.~minimum) at $x_0$, it holds $\det \D^2 \varphi(x_0) \geq (\text{resp.}\leq)~f(x_0)$. If $u$ is a viscosity sub- and supersolution to \eqref{def:det=f}, then $u$ is called viscosity solution to \eqref{def:det=f}.
	For $g \in C(\overline{\Omega})$, $u$ is a viscosity solution to the Dirichlet problem \eqref{pr:Monge-Ampere} if additionally $u = g$ on $\partial \Omega$.
\end{definition}
The two preceding concepts of weak solutions coincide for $f \in C(\Omega)$.
\begin{theorem}[existence and uniqueness]\label{thm:existence-uniqueness-Monge-Ampere}
	Suppose that \Cref{assumption:structure} holds, then there exists a unique Alexandrov solution $u \in C(\overline{\Omega})$ 
	to the Monge--Amp\`ere equation \eqref{pr:Monge-Ampere} with (a)--(c).
	\begin{enumerate}
		\item[(a)] (viscosity) If $f \in C(\Omega)$, then $u$ is the unique viscosity solution to \eqref{pr:Monge-Ampere}.
		\item[(b)] (classical) If $f \in C^{0,\alpha}(\Omega)$, $0 < \lambda \leq f \leq \Lambda$, and $g \in C^{1,\beta}(\overline{\Omega})$ with positive constants $0 < \alpha,\beta < 1$ and $0 < \lambda \leq \Lambda$, then $u \in C(\overline{\Omega}) \cap C^{2,\alpha}_{\text{loc}}(\Omega)$. Given $R, r > 0$ such that $B(x_0,r) \subset \Omega \subset B(x_0,R)$ for some $x_0 \in \Omega$, then, for all compact subsets $\Omega' \Subset \Omega$ of $\Omega$,
		\begin{align}
			\|u\|_{C^{2,\alpha}(\Omega')} \leq C_1(\lambda,\Lambda,r,R,f,g,\alpha,\dist(\Omega',\partial \Omega)).
			\label{ineq:interior-C2-estimate}
		\end{align}
		\item[(c)] (Pogorelov) If $f \in C^{2,\alpha}(\Omega)$ with $\alpha \in (0,1)$, $f > 0$ in $\overline{\Omega}$, and $g \equiv 0$, then $u \in C(\overline{\Omega}) \cap C^{4,\alpha}_{\text{loc}}(\Omega)$ and, for all $x \in \Omega$,
		\begin{align}
			\dist(x,\partial \Omega)^8 |\D^2 u(x)| \leq C_2(r,R,\|\log f\|_{C^2(\overline{\Omega})}).
			\label{ineq:Pogorelov}
		\end{align}
	\end{enumerate}
\end{theorem}
\begin{proof}
	The existence and uniqueness of Alexandrov solutions to \eqref{pr:Monge-Ampere} are due to Alexandrov \cite{Alexandrov1958}, cf.~also \cite[Corollary 2.11, Theorem 2.13--2.14]{Figalli2017} or \cite[Theorem 1.6.2]{Gutierrez2016}.
	If $f \in C(\Omega)$ in (a), then Alexandrov and convex viscosity solutions to \eqref{pr:Monge-Ampere} coincide \cite[Proposition 1.3.4]{Gutierrez2016}. 
	Notice that $f > \lambda$ and $g \in C^{1,\beta}(\overline{\Omega})$ imply strict convexity of $u$ \cite[Corollary 4.11]{Figalli2017}. In this case,
	the interior $C^{2,\alpha}$ estimate due to Caffarelli \cite{Caffarelli1990} proves (b), cf.~\cite[Corollary 4.43]{Figalli2017}.
	If $f \in C^{2,\alpha}(\Omega)$ and $f > 0$ in $\overline{\Omega}$, then $u \in C^{4,\alpha}_\mathrm{loc}(\Omega)$ \cite[Theorem 3.10]{Figalli2017} and \cite{Pogorelov1971} implies \eqref{ineq:Pogorelov}, cf.~\cite[Theorem 3.9]{Figalli2017}.
\end{proof}

\subsection{Hamilton--Jacobi--Bellman formulation}
The concept of viscosity solution to fully nonlinear PDEs arises from a comparison principle. Let $G \in C(\Omega \times \mathbb{S})$ satisfy the ellipticity condition
\begin{align}
	G(x,M) \leq G(x,N) \quad\text{for all } N \leq M.
	\label{ineq:ellipticity}
\end{align}

\begin{definition}[viscosity solution]\label{def:viscosity}
	A function $u \in C(\Omega)$ is called viscosity subsolution (resp.~supersolution) to
	\begin{align}
		G(x,\D^2 u(x)) = 0 \text{ in } \Omega \label{def:G=0}
	\end{align}
	if, for all $x_0 \in \Omega$ and $\varphi \in C^2(\Omega)$ such that $u - \varphi$ has a local maximum (resp.~minimum) at $x_0$, it holds $G(x_0, \D^2 \varphi(x_0)) \leq (\text{resp.}\geq)~0$.
	If $u$ is a viscosity sub- and supersolution to \eqref{def:G=0}, then $u$ is called viscosity solution to \eqref{def:G=0}.
	If additionally $u = g$ on $\partial \Omega$ for $g \in C(\overline{\Omega})$, then $u$ is a  viscosity solution to the Dirichlet problem
	\begin{align*}
		G(x,\D^2 u(x)) = 0 \text{ in } \Omega \quad\text{and}\quad u = g \text{ on } \partial \Omega.
	\end{align*}
\end{definition}

\begin{remark}[convexity]
The choice $G(x,M) \coloneqq f(x) - \det M$ is not elliptic on the
whole set $\mathbb{S}$ and so, \Cref{def:convex-viscosity} assumes 
the convexity of $u$ and of the test functions $\varphi$.
This problem does not arise for the HJB formulation with $G:=F$ 
from \eqref{def:F}.
\end{remark}

Notice that if $u \in C^2(\Omega)$ with $G(x,\D^2 u(x)) \leq (\text{resp.}\geq)~0$ in $\Omega$, then $u$ is viscosity subsolution (resp.~supersolution) to \eqref{def:G=0} by \Cref{def:viscosity}.
The equivalence of \eqref{pr:Monge-Ampere} and \eqref{pr:Monge-Ampere-HJB} below is established in \cite{Krylov1987} for classical solutions $u \in C^2(\Omega)$ and extended to viscosity solutions in \cite[Theorem 3.3, Theorem 3.5]{FengJensen2017}.
\begin{theorem}[equivalence]
	Let \Cref{assumption:structure} hold. If $f \in C(\Omega)$, then $u \in C(\overline{\Omega})$ is a convex viscosity solution to \eqref{pr:Monge-Ampere} in the sense of \Cref{def:convex-viscosity} if and only if $u$ is a viscosity solution to \eqref{pr:Monge-Ampere-HJB} in the sense of \Cref{def:viscosity}.
\end{theorem}

\section{Regularization}\label{sec:regularization}
This section is devoted to the convergence analysis of the regularized PDE introduced below.
\subsection{Regularized HJB formulation}
Given a regularization parameter $0 < \varepsilon \leq 1/2$, 
define 
$\mathbb{S}(\varepsilon) \coloneqq 
    \{A \in \mathbb{S}(0): 
       \varepsilon I_{2\times 2}\leq A\}
$.
Given $f \in L^1(\Omega)$ with $f \geq 0$ and $g \in H^2(\Omega)$,
the regularized PDE seeks a strong solution $u_\varepsilon \in H^2(\Omega)$ to
\begin{align}
	F_\varepsilon(f;x,\D^2 u_\varepsilon(x)) = 0 \text{ a.e.~in } \Omega \quad\text{and}\quad u_\varepsilon = g \text{ on } \partial \Omega
	\label{pr:HJB-regularized}
\end{align}
with, for all $x \in \Omega$ and $M \in \M$,
\begin{align}
	F_\varepsilon(f;x,M) \coloneqq \sup_{A \in \mathbb{S}(\varepsilon)} (-A:M + 2\sqrt{f(x)\det A})
	\label{def:F-regularized}
\end{align}
The difference of \eqref{pr:HJB-regularized}--\eqref{def:F-regularized}
compared to
\eqref{pr:Monge-Ampere-HJB}--\eqref{def:F} is that the parameter
space $\mathbb S(0)$ is replaced by the subset of those
matrices $A$ that have eigenvalues not less than $\varepsilon$.
We remark that the supremum
on the right-hand side of \eqref{def:F-regularized} 
does not change if $\mathbb{S}(\varepsilon)$ 
is replaced by the subspace of matrices
in $\mathbb{S}(\varepsilon)$
with rational coefficients.
Hence, for all $v \in H^2(\Omega)$, $F_\varepsilon(f;\cdot,\D^2 v) : \Omega \to \R$ is a measurable function \cite[Theorem 9.5]{Heinz2001}; the relations
$|A| \leq 1$ and $f \in L^1(\Omega)$ 
imply $F_\varepsilon(f;\cdot,\D^2 v) \in L^2(\Omega)$.
Since $H^2(\Omega) \subset C^{0,\vartheta}(\overline{\Omega})$ 
for all $0 < \vartheta < 1$ 
by the Sobolev embedding
\cite[Theorem 4.12 (II)]{AdamsFournier2003},
the Dirichlet boundary condition in
\eqref{pr:HJB-regularized} is enforced pointwise.
The subsequent two lemmas provide fundamental properties of the PDE 
\eqref{pr:HJB-regularized}.
\begin{lemma}[convexity and uniform ellipticity of $F_\varepsilon$]
	\label{lem:properties-F}
	Given $0 < \varepsilon \leq 1/2$. If $f \in C(\Omega)$,
	then, for all $x\in\Omega$, $F_\varepsilon(f;x,\cdot) : \mathbb{S} \to \R$ 
	is convex
	and uniformly elliptic in the sense that, for all 
	$M,N \in \mathbb{S}$ with $N \geq 0$,
	\begin{align}
		\varepsilon |N| \leq F_\varepsilon(f; x, M) - F_\varepsilon(f; x, M + N) \leq (\varepsilon^2 + (1-\varepsilon)^2)^{1/2}|N|.
		\label{ineq:uniform-ellipticity}
	\end{align}
\end{lemma}
\begin{proof}
The asserted convexity readily follows
from the inequality 
$\sup (X + Y) \leq \sup X + \sup Y$
for any pair of sets $X,Y\subset\mathbb R$ of real numbers.
The elementary estimate
$\sup (-X) - \sup (-(X+Y)) \leq \sup Y$
shows that any $M, N \in \mathbb{S}$ satisfy
\begin{align*}
	F_\varepsilon(f; x, M) - F_\varepsilon(f; x, M + N)
	\leq \sup_{A \in \mathbb{S}(\varepsilon)} A:N.
\end{align*}
Since any matrix $A \in \mathbb{S}(\varepsilon)$ satisfies
$A : N  \leq |A| |N| 
\leq (\varepsilon^2 + (1-\varepsilon)^2)^{1/2}|N|$,
this shows the
asserted upper bound in \eqref{ineq:uniform-ellipticity}.
The proof of the lower bound in \eqref{ineq:uniform-ellipticity} is partly contained in \cite[p.~50--51]{Krylov1987}. We outline the main arguments below.
The inequality $\sup (X + Y) \leq \sup X + \sup Y$ implies
\begin{align}\label{ineq:proof-uniform-ellipticity-LHS}
	F_\varepsilon(f; x, M + N) 
	\leq F_\varepsilon(f; x, M) 
	+ \sup_{A \in \mathbb{S}(\varepsilon)} (-A:N).
\end{align}
Let $0 \leq \varrho_1(N) \leq \varrho_2(N)$ denote 
the non-negative eigenvalues of the matrix $N$. Since eigenvalues are invariant under similarity transformations, it holds
\begin{align}
	-\sup_{A \in \mathbb{S}(\varepsilon)} (-A:N) = \inf_{A \in \mathbb{S}(\varepsilon)} A:N = \inf_{A \in \mathbb{S}(\varepsilon)} A:\begin{pmatrix}
		\varrho_1(N) & 0\\0 & \varrho_2(N)
	\end{pmatrix}.
	\label{ineq:proof-ellipticity-lower-bound}
\end{align}
In two space dimensions, the diagonal elements of any positive definite matrix $A \in \mathbb{S}(\varepsilon)$ is in the interval $[\varepsilon, 1 -\varepsilon]$. Hence, the infima in \eqref{ineq:proof-ellipticity-lower-bound} are equal to $\inf_{\varepsilon \leq t \leq 1 - \varepsilon} (t \varrho_1(N) + (1 - t)\varrho_2(N))$. Since $0 \leq \varrho_1(N) \leq \varrho_2(N)$, the infimum is attained at $t = 1 - \varepsilon$ with the lower bound $\varepsilon (\varrho_1(N) + \varrho_2(N)) \leq \inf_{A \in \mathbb{S}(\varepsilon)} A:N$. This, the inequality $|N| = \sqrt{\varrho_1(N)^2 + \varrho_2(N)^2} \leq \varrho_1(N) + \varrho_2(N)$, and \eqref{ineq:proof-uniform-ellipticity-LHS} prove the asserted lower bound in	\eqref{ineq:uniform-ellipticity}.
\end{proof}

The following lemma states a comparison principle
for viscosity sub- and supersolutions
defined in \Cref{def:viscosity}.

\begin{lemma}[comparison principle]\label{lem:comparison-principle}
Let $f \in C^{0,\alpha}(\Omega)$ 
with $0 < \alpha < 1$ and let $v_* \in C(\overline{\Omega})$
and $v^* \in C(\overline{\Omega})$ 
be a sub- and supersolution to 
$F_\varepsilon(f;x,\D^2 v(x)) = 0$ in $\Omega$.
If $v_* \leq v^*$ on $\partial \Omega$, then $v_* \leq v^*$ in $\overline{\Omega}$.
\end{lemma}
\begin{proof}
We show the lemma by verifying that $F_\varepsilon$
satisfies the structural assumptions of the
general result \cite[Theorem 3.9]{Koike2004}.
The elementary estimate $|\sup X -\sup Y| \leq \sup |X-Y|$
for bounded sets $X,Y$ of real numbers
together with
$\det A \leq 1/4$ for all $A \in \mathbb{S}(0)$ 
and $f \in C^{0,\alpha}(\Omega)$
show, for any $x,y\in\Omega$ and $M\in\mathbb S$, that
\begin{align*}
	|F_\varepsilon(f;x,M) - F_\varepsilon(f;y,M)| 
	\leq |\sqrt{f(x)} - \sqrt{f(y)}|
	\leq \sqrt{|f(x) - f(y)|} \leq C|x-y|^{\alpha/2}
\end{align*}
for $C \coloneqq \|f\|_{C^{0,\alpha}(\Omega)}^{1/2}$.
This, the uniform ellipticity from \Cref{lem:properties-F} above,
and \cite[Proposition 3.8]{Koike2004} prove the critical assumption
(3.21) in \cite[Theorem 3.9]{Koike2004},
which implies the comparison principle.
\end{proof}

\subsection{Analysis of the regularized PDE}\label{sec:analysis-regularized-PDE}
The regularized PDE \eqref{pr:HJB-regularized}
is uniformly elliptic, which
is sufficient for its well-posedness in two space dimensions.
Results of this type were established
by \cite{SmearsSueli2014}. Following their approach,
we define the scaling function, for any positive definite matrix $A$,
\begin{equation*}
 \gamma(A) \coloneqq \frac{\tr A}{|A|^2}.
\end{equation*}
The scaled version of $F_\varepsilon$ reads
\begin{align}
	F_{\gamma,\varepsilon}(f; x, M) 
	\coloneqq \sup_{A \in \mathbb{S}(\varepsilon)}
	\big(\gamma(A)(-A:M + 2\sqrt{f(x)\det A})\big)
	\label{def:F-gamma-eps}
\end{align}
for all $x \in \Omega$ and $M \in \M$.
The following theorem utilizes the constants
$$ 
	2\varepsilon \leq \delta(\varepsilon) \coloneqq 2\varepsilon(1 - \varepsilon)
	/(\varepsilon^2 + (1-\varepsilon)^2) ~\leq~ 1
	\quad\text{and}\quad
	C(\varepsilon) \coloneqq (1 + \sqrt{2})/(1 - \sqrt{1 - \delta(\varepsilon)}) .
$$
We note that $C(\varepsilon)$ 
satisfies the scaling $C(\varepsilon)\approx\varepsilon^{-1}$
and is monotonically increasing if $\varepsilon$ decreases.
It is essential that any $A\in\mathbb S(\varepsilon)$
satisfies \eqref{e:cordes} with $\delta \coloneqq \delta(\varepsilon)$ defined
above
\cite[p.~163]{NeilanSalgadoZhang2017}.

\begin{theorem}[strong solution to regularized PDE]
	\label{thm:existence-uniqueness-regularized}
	Let $\Omega$ be convex and given
	$0 < \varepsilon \leq 1/2$,  $f \in L^1(\Omega)$ with $f \geq 0$, 
	and $g \in H^2(\Omega)$.
	Then (a)--(d) hold. 
	\begin{enumerate}
		\item[(a)] (stability) 
		Let $u_\varepsilon \in H^2(\Omega)$ solve
		\eqref{pr:HJB-regularized}
		and let  $\widetilde u_\varepsilon \in H^2(\Omega)$ solve 
		\eqref{pr:HJB-regularized} with $f$
		replaced by $\widetilde{f} \in L^1(\Omega)$ with $\widetilde{f} \geq 0$
		and $g$ replaced by $\widetilde{g} \in H^2(\Omega)$.
		Then
		\begin{align*}
			\|\Delta (u_\varepsilon - \widetilde u_\varepsilon)\|_{L^2(\Omega)}
			\leq 
			C(\varepsilon)(\|f - \widetilde{f}\|_{L^1(\Omega)}^{1/2}
			+ \|g - \widetilde{g}\|_{H^2(\Omega)}).
		\end{align*}
		In particular, $\|u_\varepsilon - \widetilde u_\varepsilon\|_{H^2(\Omega)} \lesssim (\|f - \widetilde{f}\|_{L^1(\Omega)}^{1/2}
			+ \|g - \widetilde{g}\|_{H^2(\Omega)})/\varepsilon$.
		\item[(b)] (existence) There exists a unique strong solution $u_\varepsilon \in H^2(\Omega)$ to \eqref{pr:HJB-regularized}.
		\item[(c)] (viscosity solution) If $f \in C(\Omega)$, then the unique strong solution $u_\varepsilon$ to \eqref{pr:HJB-regularized} is the unique viscosity solution to \eqref{pr:HJB-regularized}.
		\item[(d)] (classical solution) If $f \in C^{0,\alpha}(\Omega)$ with $0 < \alpha < 1$, then $u_\varepsilon \in C(\overline{\Omega}) \cap C^{2,\kappa}_{\text{loc}}(\Omega)$ is the unique classical solution to \eqref{pr:HJB-regularized}. 
    	The constant $0 < \kappa < 1$ solely depends on $\alpha$ and $\varepsilon$.
	\end{enumerate}
\end{theorem}
\begin{proof}[Proof of \Cref{thm:existence-uniqueness-regularized}(a)]
	The proofs of (a)--(b) can be found in \cite{SmearsSueli2014} 
	for homogeneous boundary data $g = 0$. 
	The arguments therein can be extended to $g \neq 0$ and are outlined
	below for the convenience of the reader. 
	The proof utilizes the following two inequalities.
	Abbreviate $e \coloneqq u_\varepsilon - \widetilde u_\varepsilon$.
	The first one reads
	\begin{align}
		|F_{\gamma,\epsilon}(f;x,M) - F_{\gamma,\epsilon}(\widetilde{f};y,N)| \leq 2|M - N| + 2|f(x) - \widetilde{f}(y)|^{1/2}
		\label{ineq:F-eps-upper-bound}
	\end{align}
	for all $x,y \in \Omega$ and $M,N \in \M$.
	This follows from elementary properties of suprema as in
	the proof of Lemma~\ref{lem:comparison-principle}
	together with the triangle inequality
	and the bounds $\gamma \leq 2$,
	$|A| \leq 1$,
	and $\det A \leq 1/4$ for all $A \in \mathbb{S}(\varepsilon)$.
	The second one
	\begin{align}
		&|F_{\gamma,\varepsilon}(\widetilde{f};x,\D^2 \widetilde{u}_\varepsilon(x)) - F_{\gamma,\varepsilon}(\widetilde{f};x,\D^2 u_\varepsilon(x)) - \Delta e(x)|\nonumber\\
		&\qquad \leq \sup_{A \in \mathbb{S}(\varepsilon)}|\gamma(A)A - I||\D^2 e(x)| \leq \sqrt{1 - \delta(\varepsilon)}|\D^2 e(x)|
		\label{ineq:pw-estimate}
	\end{align}
	for a.e.~$x \in \Omega$ is established in 
	\cite[Lemma~1]{SmearsSueli2014}
	and follows from similar arguments as \eqref{ineq:F-eps-upper-bound}
	because the bound \eqref{e:cordes}
	provides $|\gamma(A)A - I| \leq \sqrt{1 - \delta(\varepsilon)}$,
	cf.\ \cite{SmearsSueli2013,SmearsSueli2014}.
	The weights $\gamma(A) > 0$ are positive for any 
	$A \in \mathbb{S}(\varepsilon)$ and so, 
	$F_\varepsilon(f; x, \D^2 u_\varepsilon(x)) 
	= 0 = F_\varepsilon(\widetilde{f}; x, \D^2 \widetilde u_\varepsilon(x))
	$ 
	implies 
	$F_{\gamma,\varepsilon}(f; x, \D^2 u_\varepsilon(x)) 
	= 0 
	= F_{\gamma,\varepsilon}
	(\widetilde{f}; x, \D^2 \widetilde u_\varepsilon(x))
	$
	a.e.~in $\Omega$. 
	This, \eqref{ineq:F-eps-upper-bound}, 
	and the Cauchy inequality lead to
	\begin{equation}\label{ineq:LHS-upper-bound}
		\begin{aligned}
			\text{LHS} 
			\coloneqq& 
			\int_\Omega
			(F_{\gamma,\varepsilon}(\widetilde{f};
			x,\D^2 \widetilde u_\varepsilon(x)) 
			- F_{\gamma,\varepsilon}(\widetilde{f};x,\D^2 u_\varepsilon(x)))
			\Delta e(x) \d{x}\\
			=& \int_\Omega (F_{\gamma,\varepsilon}(f;x,\D^2 u_\varepsilon(x))
			- F_{\gamma,\varepsilon}(\widetilde{f};x,\D^2 u_\varepsilon(x)))
			\Delta e(x) \d{x}\\
			\leq&
			~2\|f - \widetilde{f}\|_{L^1(\Omega)}^{1/2}\|\Delta e\|_{L^2(\Omega)}.
		\end{aligned}
	\end{equation}
	On the other hand, \eqref{ineq:pw-estimate} and the Cauchy inequality confirm
	\begin{align}
		\|\Delta e\|_{L^2(\Omega)}^2 - \sqrt{1 - \delta(\varepsilon)}\|\D^2 e\|_{L^2(\Omega)} \|\Delta e\|_{L^2(\Omega)} \leq \mathrm{LHS}.
		\label{ineq:LHS-lower-bound}
	\end{align}
	The Miranda–Talenti estimate from \cite[Chapter 3]{Grisvard1985} 
        holds for convex domains $\Omega$ with $C^{2}$ boundary, but is extended to general 
        bounded convex domains in \cite[Theorem 2]{SmearsSueli2013}. This 
	allows for control over the 
	$H^2$ norm by the $L^2$ norm of the Laplacian if the function
	vanishes on the boundary. This, multiple applications of the 
	triangle inequality, and 
	$\|\Delta e\|_{L^2(\Omega)} \leq \sqrt{2}\|e\|_{H^2(\Omega)}$ in 
	two-dimensional domains provide
	\begin{align*}
		&\|\D^2 e\|_{L^2(\Omega)} \leq \|\D^2 (e + \widetilde{g} - g)\|_{L^2(\Omega)} + \|\widetilde{g} - g\|_{H^2(\Omega)}\\
		&\quad\leq \|\Delta (e + \widetilde{g} - g)\|_{L^2(\Omega)} + \|\widetilde{g} - g\|_{H^2(\Omega)} \leq \|\Delta e\|_{L^2(\Omega)} + (1+\sqrt{2})\|\widetilde{g} - g\|_{H^2(\Omega)}.
	\end{align*}
	This, \eqref{ineq:LHS-upper-bound}--\eqref{ineq:LHS-lower-bound},
        the Miranda–Talenti estimate
	conclude the proof of (a).\\[1em]
	\emph{Proof of \Cref{thm:existence-uniqueness-regularized}(b).}
	Let $H \coloneqq V \cap H^2(\Omega)$ with $V = H^1_0(\Omega)$,
	endowed with the norm 
	$\|\cdot\|_{H^2(\Omega)} \approx \|\Delta \cdot\|_{L^2(\Omega)}$.
	We define the semilinear form
	$A : H \to H^*$ by
	\begin{align*}
		\langle A v, w\rangle \coloneqq 
		-\int_\Omega F_{\gamma,\varepsilon}
		(f;x,\D^2(v + g)(x))\Delta w(x) \d{x}
		\quad\text{for all } v,w \in H.
	\end{align*}
	Notice that \eqref{ineq:LHS-lower-bound} holds for
	the replacements
	$u_\varepsilon \coloneqq v + g$ and 
	$\widetilde u_\varepsilon \coloneqq w + g$
	in the definition of LHS.
	This and the Miranda--Talenti estimate 
	$\|\D^2(v - w)\|_{L^2(\Omega)} \leq \|\Delta(v - w)\|_{L^2(\Omega)}$
	\cite[Lemma 2]{SmearsSueli2013} prove the strong monotonicity of 
	$A$ in the sense that 
	$(1 - \sqrt{1 - \delta(\varepsilon)})\|\Delta(v - w)\|_{L^2(\Omega)}^2 
	\leq  \langle Av, v - w\rangle - \langle Aw, v - w\rangle$. 
	On the other hand, \eqref{ineq:F-eps-upper-bound} implies
	the Lipschitz continuity
	\begin{align*}
		|\langle Av,z\rangle - \langle A w,z\rangle|
		\leq 2\|\D^2(v - w)\|_{L^2(\Omega)}\|\Delta z\|_{L^2(\Omega)} .
	\end{align*}
	The Browder--Minty theorem thus shows the existence and uniqueness 
	of $v_\varepsilon \in V$ such that $A v_\varepsilon = 0$.
	Since $\Delta V = L^2(\Omega)$ due to the elliptic regularity
	on convex domains \cite{Grisvard1985},
	this implies 
	$F_{\varepsilon}(f; x,\D^2 (v_\varepsilon + g)(x)) 
	= F_{\gamma,\varepsilon}(f; x,\D^2 (v_\varepsilon + g)(x)) = 0$
	a.e.~in $\Omega$.
	Hence, $u_\varepsilon \coloneqq v_\varepsilon + g \in H^2(\Omega)$ 
	is a strong solution to \eqref{pr:HJB-regularized}.
	Uniqueness of $u_\varepsilon$ follows from the stability in (a).\\[1em]
	\emph{Proof of \Cref{thm:existence-uniqueness-regularized}(c).}
	The proof utilizes classical PDE theory and the stability of viscosity
	solutions under $C^0$ perturbations \cite{Katzourakis2015}.
	Given $x_0 \in \Omega$ 
	and $r>0$ such that the open ball $B \coloneqq B(x_0,r)$ of radius $r$ is inscribed in $\Omega$, that is $\overline{B} \Subset \Omega$.
	For all $j \in \mathbb{N}$, let 
	$u_\varepsilon^{(j)} \in C^\infty(\overline{B})$ 
	and $f_j \in C^\infty(\overline{B})$ be smooth 
	approximations of $u_\varepsilon$ and $f$ in $\overline{B}$
	with 
	$\lim_{j \to \infty}
	(\|u_\varepsilon - u_\varepsilon^{(j)}\|_{H^2(B)} 
	+ \|f - f_j\|_{L^\infty(B)}) = 0
	$.
	Consider the PDE
	\begin{align}
		F_\varepsilon(f_j;x,\D^2 v_j(x)) = 0 \text{ in } B 
		\quad\text{and}\quad 
		v_j = u_\varepsilon^{(j)} \text{ on } \partial B .
		\label{pr:HJB-regularized-ball}
	\end{align}
	The problem \eqref{pr:HJB-regularized-ball} satisfies the structural assumptions (F0)--(F4) and (B1)--(B2) in \cite{Safonov1988}. Hence, there exists a unique classical solution $v_j \in C^{2}(\overline{B})$ to \eqref{pr:HJB-regularized-ball} \cite[Theorem 1.2]{Safonov1988}. In particular, $v_j$ satisfies \eqref{pr:HJB-regularized-ball} pointwise in $B$.
	The stability in (a) proves
	\begin{align*}
		\|u_\varepsilon - v_j\|_{H^2(B)} \lesssim (\|u_\varepsilon - u_\varepsilon^{(j)}\|_{H^2(B)} + \|f - f_j\|_{L^1(B)}^{1/2})/\varepsilon.
	\end{align*}
	Since the right-hand side vanishes in the limit $j \to \infty$, the Sobolev embedding $H^2(B) \subset C^{0,\vartheta}(\overline{B})$ for all $\vartheta \in (0,1)$ \cite[Theorem 4.12 (II)]{AdamsFournier2003} shows that $v_j$ converges uniformly to $u_\varepsilon$ in $\overline{B}$ as $j \to \infty$. The stability of viscosity solutions under $C^0$ perturbations \cite[Chapter 3, Thm.~2]{Katzourakis2015} concludes that $u_\varepsilon$ is a viscosity solution to $F_\varepsilon(f; x, \D^2 u_\varepsilon(x)) = 0$ in $B$ and so, $u_\varepsilon$ is a viscosity solution to \eqref{pr:HJB-regularized}.\\[1em]
	\emph{Proof of \Cref{thm:existence-uniqueness-regularized}(d).}
	If $f \in C^{0,\alpha}(\Omega)$, then the PDE \eqref{pr:HJB-regularized} satisfies the structural assumptions of 
        \cite[Theorem 1.1]{Safonov1988}. This provides $u_\varepsilon \in C(\overline{\Omega}) \cap C^{2,\kappa}_{\text{loc}}(\Omega)$
       with a constant $\kappa$ that depends on $\alpha$, $\varepsilon$, and $\Omega$.
	Since $F(f;x,\cdot)$ is convex, the perturbation theory of \cite{CaffarelliCabre1995} 
       applies to the PDE \eqref{pr:HJB-regularized} and leads to the same regularity result.
\end{proof}

An implication of the comparison principle in 
\Cref{lem:comparison-principle} is the following monotonicity
with respect to the regularization parameter.

\begin{lemma}[monotonicity]\label{lem:monotonicity}
Let $f \in C^{0,\alpha}(\Omega)$ with $\alpha \in (0,1)$, $f \geq 0$,
and $g \in H^2(\Omega)$.
Given two regularization parameters
$0 < \varepsilon \leq \zeta \leq 1/2$,
let
$u_\varepsilon, u_\zeta \in H^2(\Omega)$ be the unique strong solutions to
\eqref{pr:HJB-regularized} and \eqref{pr:HJB-regularized} (with $\varepsilon$ replaced by $\zeta$).
Then
\begin{align}
	u \leq u_\varepsilon \leq u_\zeta \leq u_{1/2} \text{ in } \overline{\Omega},
	\label{ineq:monotonicity}
\end{align}
where $u$ is the viscosity solution to the Monge--Amp\`ere equation and
$
u_{1/2} \in H^2(\Omega)
$ 
denotes the solution to the Poisson model problem $\Delta u_{1/2} = 2\sqrt{f}$ in $\Omega$ 
with $u_{1/2} = g$ on $\partial \Omega$.
\end{lemma}
\begin{proof}
Notice from \Cref{thm:existence-uniqueness-regularized}(d) that $u_\varepsilon, u_\zeta, u_{1/2} \in H^2(\Omega) \cap C^2(\Omega)$ are classical solutions, but $u$ may not be in $C^{2,\alpha}_{\mathrm{loc}}(\Omega)$, cf.~\Cref{thm:existence-uniqueness-Monge-Ampere}(b) for sufficient conditions for this.
Since $u_\zeta$ solves \eqref{pr:HJB-regularized} 
with $\varepsilon \coloneqq \zeta$, 
it satisfies $F_{\zeta}(f;x,\D^2 u_\zeta(x)) = 0$ for all $x \in \Omega$.
This and the inclusion
$\mathbb{S}(\zeta) \subset \mathbb{S}(\varepsilon)$
imply 
$$
	F_{\varepsilon}(f;x,\D^2 u_\zeta(x)) 
	\geq F_{\zeta}(f;x,\D^2 u_\zeta(x)) = 0
	\quad\text{in }\Omega.
$$
In particular, $u_\zeta$ is viscosity supersolution to \eqref{pr:HJB-regularized}. Notice that $u_{1/2}$ solves \eqref{pr:HJB-regularized} with $\varepsilon \coloneqq 1/2$. Hence, the comparison principle from \Cref{lem:comparison-principle} proves
\begin{align}
	u_\varepsilon \leq u_\zeta \leq u_{1/2} \text{ in } \overline{\Omega}.
	\label{ineq:proof-monotonicity}
\end{align}
Similarly,
the inclusion $\mathbb{S}(\varepsilon) \subset \mathbb{S}(0)$ 
implies, for any test function $\varphi \in C^2(\Omega)$ such that $u - \varphi$ has a local maximum at $x$, that
$$
	F_\varepsilon(f;x,\D^2 \varphi(x)) \leq F(f;x,\D^2 \varphi(x)) \leq 0
	\quad\text{for any } x \in \Omega.
$$
Hence, $u$ is a viscosity subsolution to \eqref{pr:HJB-regularized}
and the comparison principle from
Lem\-ma~\ref{lem:comparison-principle} provides
$u \leq u_\varepsilon$ in $\overline{\Omega}$. This and \eqref{ineq:proof-monotonicity} conclude the proof.
\end{proof}

\section{Convergence analysis}\label{sec:convergence}
This section establishes the uniform convergence
of $u_\varepsilon$ to $u$ in $\Omega$ as $\varepsilon \to 0$.
The proof departs from the convergence result for sufficiently
smooth right-hand side $f$ below.

\begin{theorem}[convergence]\label{thm:convergence}
	Suppose that Assumption~\ref{assumption:structure} holds.
	Let $f \in C^{0,\alpha}(\Omega)$,
	$0 < \lambda \leq f \leq \Lambda$,
	and $g \in H^2(\Omega) \cap C^{1,\beta}(\overline{\Omega})$
	with positive constants
	$0 < \alpha,\beta < 1$ and $0 < \lambda \leq \Lambda$ be given.
	Let
	$u \in C(\overline{\Omega}) \cap C^{2,\alpha}_{\mathrm{loc}}(\Omega)$
	be the unique classical solution to \eqref{pr:Monge-Ampere} from
	\Cref{thm:existence-uniqueness-Monge-Ampere}(b).
	\begin{enumerate}
		\item[(a)] For any sequence
		$0<(\varepsilon_j)_{j \in \mathbb{N}}\leq 1/2$ with $\lim_{j \to \infty} \varepsilon_j = 0$, the sequence $(u_{\varepsilon_j})_{j \in \mathbb{N}}$ of strong solutions $u_{\varepsilon_j} \in H^2(\Omega)$ to \eqref{pr:HJB-regularized} with $\varepsilon \coloneqq \varepsilon_j$ converges uniformly to $u$ in $\Omega$ as $j \to \infty$.
		\item[(b)] If $g\equiv 0$, $f \in C^{2,\alpha}(\Omega)$, and
		$f > 0$ in $\overline{\Omega}$,
		then, for all $0 < \varepsilon \leq 1/2$, the strong solution $u_\varepsilon \in H^2(\Omega)$
		to \eqref{pr:HJB-regularized}
		satisfies
		\begin{align}
			\|u - u_\varepsilon\|_{L^\infty(\Omega)}
			\lesssim \varepsilon^{1/32}.
			\label{ineq:rate-eps}
		\end{align}
	\end{enumerate}
\end{theorem}

The proof of \Cref{thm:convergence} relies on the following
elementary lemma.
\begin{lemma}[effect of regularization]\label{lem:regularization}
	Given a symmetric matrix $M \in \mathbb{S}$, a positive number $\zeta > 0$, and a positive parameter $0 < \varepsilon \leq 1/2$. Assume that
	\begin{align}
		\max_{A \in \mathbb{S}(0)} (-A:M + \zeta \sqrt{\det A}) = 0.
		\label{eq:HJB-max=0}
	\end{align}
	Then the maximum of \eqref{eq:HJB-max=0} is attained in the subset $\mathbb{S}(\varepsilon) \subset \mathbb{S}(0)$ if and only if $|M|^2 \leq ((1-\varepsilon)^2 + \varepsilon^2) \zeta^2/(4 \varepsilon(1-\varepsilon))$.
\end{lemma}
\begin{proof}
Let $M \in \mathbb{S}$ with the eigenvalues $\varrho_1$ and $\varrho_2$ be given. Define $\Phi(A) \coloneqq -A:M + \zeta\sqrt{\det A}$ for any $A \in \mathbb{S}(0)$ and $\psi(s) \coloneqq ((1-s)^2 + s^2) \zeta^2/(4s(1-s))$ for all $0 < s < 1$.
We argue as in \eqref{ineq:proof-ellipticity-lower-bound} and assume, without loss of generality, that $M$ is a diagonal matrix.
It is well-known \cite[p.~51]{Krylov1987} that the maximum in \eqref{eq:HJB-max=0} is attained in the subset of diagonal matrices $A \in \mathbb{S}(0)$. Thus, $\eqref{eq:HJB-max=0}$ holds if and only if $\max_{t \in [0,1]} \varphi(t) = 0$
for the scalar-valued function $\varphi : [0,1] \to \R$ with $\varphi(t) \coloneqq - t \varrho_1 - (1 - t) \varrho_2 + \zeta\sqrt{t(1 - t)}$.
The derivatives of $\varphi$ read as
$$
\varphi'(t) = \varrho_2 - \varrho_1 
        + \frac{(1-2t)\zeta}{2\sqrt{t(1-t)}}
\quad\text{and}\quad
\varphi''(t) = - \frac{\zeta}{4(t(1-t))^{3/2}}.
$$
The property $\varphi''<0$ in $(0,1)$ shows that
$\varphi'$ is strictly monotonically decreasing with
$\lim_{t \to 0} \varphi'(t) = \infty$ 
and $\lim_{t \to 1} \varphi'(t) = - \infty$.
Thus, the unique maximizer $\overline t$ of $\varphi$
in $[0,1]$ is an interior point
$\overline{t} \in (0,1)$ with
$\varphi(\overline{t}) = 0$ and $\varphi'(\overline{t}) = 0$.
This verifies
\begin{align*}
 \varrho_1 
 = \frac{(1 - \overline{t}) \zeta}{2\sqrt{\overline{t}(1-\overline{t})}}
 \quad\text{and}\quad 
\varrho_2 
= \frac{\overline{t} \zeta}{2\sqrt{\overline{t}(1-\overline{t})}}.
\end{align*}
We prove now that 
$\overline{t} \in [\varepsilon, 1-\varepsilon]$ 
if and only if
$
\varrho_1^2 + \varrho_2^2 
 \leq \psi(\varepsilon)
$.
The scalar-valued function $\psi$ with the derivatives
\begin{align*}
	\psi'(s) = \frac{(2s - 1) \zeta^2}{4s^2(1-s)^2}\quad\text{and}\quad
	\psi''(s) = \frac{(1 - 3 s + 3 s^2) \zeta^2}{2s^3(1-s)^3}
\end{align*}
is strictly convex because $\psi''(s)>0$ for all $s \in (0,1)$.
Since $\psi'(s)$
vanishes if and only if $s = 1/2$, the function $\psi$ is monotonically decreasing in $(0,1/2]$,
monotonically increasing in $[1/2,1)$,
and symmetric in the sense that 
$\psi(s) = \psi(1-s)$.
Hence, $\overline{t} \in [\varepsilon, 1-\varepsilon]$ 
if and only if $\psi(\overline{t}) \leq \psi(\varepsilon)$.
In particular, $|M|^2 \leq \psi(\varepsilon)$ implies $0 = \varphi(\overline{t}) = \max_{t \in [\varepsilon,1-\varepsilon]} \varphi(t) \leq \max_{A \in \mathbb{S}(\varepsilon)} \Phi(A)$ and so, \eqref{eq:HJB-max=0} concludes $\max_{A \in \mathbb{S}(\varepsilon)} \Phi(A) = 0$. On the other hand, 
if there exists $A \in \mathbb{S}(\varepsilon)$ such that $\Phi(A) = 0$, 
then \cite[Lemma 3.2.1]{Krylov1987} proves that $\widetilde{A} \coloneqq \big(\begin{smallmatrix} a_{11} & 0\\ 0 & a_{22}\end{smallmatrix}\big) \in \mathbb{S}(\varepsilon)$ with the diagonal elements $a_{11}, a_{22}$ of $A$ satisfies $0 = \Phi(A) \leq \Phi(\widetilde{A}) \leq \max_{t \in [\varepsilon,1-\varepsilon]} \varphi(t)$. Thus, \eqref{eq:HJB-max=0} shows $\max_{t \in [\varepsilon,1-\varepsilon]} \varphi(t) = 0$ and so, it follows $|M|^2 \leq \psi(\varepsilon)$.
\end{proof}

\begin{remark}[no need for regularization]
	Suppose that $f$ and $g$ satisfy the assumptions of \Cref{thm:convergence}.
	If $\|\D^2 u\|_{L^\infty(\Omega)}$ is bounded,
	then \Cref{lem:regularization} shows that
	$u$ is strong solution to the regularized problem
	\eqref{pr:HJB-regularized} for sufficiently small $\varepsilon$.
	Hence, asymptotically the regularization has no effect.
	A similar result has been proven in \cite[Theorem 5.2]{BrennerKawecki2021}.
\end{remark}

A key observation from \Cref{lem:regularization} (applied to $\zeta:=2\sqrt{f(x)}$) is that 
$F_{\varepsilon}(f;x,\allowbreak\D^2 u(x)) = 0$ if $|\D^2 u(x)|^2$ is bounded 
by $C_3(\varepsilon)f(x)/\varepsilon$ for
\begin{align}
	C_3(\varepsilon) 
	\coloneqq ((1-\varepsilon)^2 + \varepsilon^2)/(1 - \varepsilon) 
	\geq 2(\sqrt{2}-1) \geq 4/5.
	\label{def:C-3}
\end{align}
This holds in any compact subset $\Omega' \Subset \Omega$ of $\Omega$
with sufficiently small $\varepsilon$ due to the interior
$C^{2,\alpha}$ estimate 
from \Cref{thm:existence-uniqueness-Monge-Ampere}(b). 
This is the point of departure of the proof 
of \Cref{thm:convergence}.

\begin{proof}[Proof of \Cref{thm:convergence}(a).]		
	Recall $u_{1/2}$ from \Cref{lem:monotonicity}.
	The function $u_{1/2} - u$ is continuous on the compact set $\overline{\Omega}$. Hence, $u_{1/2} - u$ is uniformly continuous in $\overline{\Omega}$. Given $\nu > 0$, there exists $s > 0$ such that
	\begin{align}
		|(u_{1/2} - u)(x) - (u_{1/2} - u)(y)| \leq \nu \text{ for all } x,y \in \overline{\Omega} \text{ with } |x - y| \leq s.
		\label{ineq:proof-convergence-uniform-continuity}
	\end{align}
	Define the compact subset $\Omega' \coloneqq \{x \in \Omega: \dist(x,\partial \Omega) \geq s\} \Subset \Omega$ of $\Omega$.
	Given $x \in \overline{\Omega \setminus \Omega'}$, there exists $y \in \partial \Omega$ such that $|x - y| \leq s$. This, \eqref{ineq:monotonicity}, and \eqref{ineq:proof-convergence-uniform-continuity} show
	\begin{align*}
		0 \leq u_{1/2}(x) - u(x)\leq \nu \text{ in } \overline{\Omega \setminus \Omega'}.
	\end{align*}
	The combination of this with \eqref{ineq:monotonicity} implies, for all $x \in \overline{\Omega \setminus \Omega'}$ and $j \in \mathbb{N}$,
	\begin{align}
		0 \leq  u_{\varepsilon_j}(x) - u(x)\leq u_{1/2}(x) - u(x)\leq \nu.
		\label{ineq:proof-uniform-convergence-boundary-layer}
	\end{align}
	Recall $0 < \lambda \leq f$. We choose $N > 0$ such that 
	$\varepsilon_j \leq 4\lambda/(5 C_1^2)$ 
	with $C_1 \coloneqq C_1(\lambda,\Lambda,r,R,f,g,\alpha,s)$ 
	from \Cref{thm:existence-uniqueness-Monge-Ampere}(b) 
	for all $j \geq N$. With this choice,
	\Cref{thm:existence-uniqueness-Monge-Ampere}(b) 
	and the definition of $C_3(\varepsilon_j)$ from \eqref{def:C-3} provide, for all $x \in \Omega'$,
	\begin{align*}
		|\D u(x)|^2 
		\leq C_1^2 
		\leq 4\lambda/(5\varepsilon_j)
		\leq C_3(\varepsilon_j) \lambda/\varepsilon_j.
	\end{align*}
	This and \Cref{lem:regularization} prove 
	$F_{\varepsilon_j}(f;x,\D^2 u(x)) = 0$ pointwise in $\Omega'$,
	whence $u$ is viscosity supersolution to 
	$F_{\varepsilon_j}(f;x,\D^2 v(x)) = 0$ in $\Omega'$.
	We recall from the monotonicity
	\eqref{ineq:monotonicity} that $0\leq u_{\varepsilon_j}-u$.
	Since, by elementary arguments,
	$
	u_{\varepsilon_j} 
	 - \max_{x \in \partial \Omega'} (u_{\varepsilon_j}(x) - u(x)) 
	 \leq u
	$ 
	holds on $\partial\Omega'$
	and the left-hand side of this estimate is
	a viscosity solution to $F_{\varepsilon_j}(f;x,\D^2 v(x)) = 0$ in $\Omega'$,
	the comparison principle from \Cref{lem:comparison-principle}
    leads to
	\begin{align}
		0 
		\leq \max_{x \in \overline{\Omega'}}(u_{\varepsilon_j}(x) - u(x)) 
		\leq \max_{x \in \partial\Omega'} (u_{\varepsilon_j}(x) - u(x)).
		\label{ineq:proof-convergence-comparison-principle}
	\end{align}
	This and \eqref{ineq:proof-uniform-convergence-boundary-layer} conclude $\|u - u_{\varepsilon_j}\|_{L^\infty(\Omega)} \leq \nu$ for all $j \geq N$ and so, $u_{\varepsilon_j}$ converges uniformly to $u$ in $\Omega$ as $j \to \infty$.\\[1em]
	\emph{Proof of \Cref{thm:convergence}(b).}
	Given $0 < \varepsilon \leq 1/2$. Recall $C_2 \coloneqq C_2(r,R,\|\log f\|_{C^2(\overline{\Omega})})$ from \Cref{thm:existence-uniqueness-Monge-Ampere}(c) and define the compact subset $\Omega' \coloneqq \{x \in \Omega: \dist(x,\partial \Omega)^{16} \geq 5 C_2^2 \varepsilon/(4\lambda)\} \Subset \Omega$ of $\Omega$.
	The Pogorelov estimate \eqref{ineq:Pogorelov} implies
	\begin{align*}
		|\D^2 u(x)|^2 \leq C_2^2/\dist(x,\partial \Omega)^{16} \leq C_3(\varepsilon)\lambda/\varepsilon \text{ for all } x \in \Omega'.
	\end{align*}
	Hence, \eqref{ineq:proof-convergence-comparison-principle} holds verbatim with $\varepsilon_j$ replaced by $\varepsilon$.
	The function $w \coloneqq 0 \in H^2(\Omega)$ satisfies $F_\varepsilon(f;x,\D^2 w(x)) > 0$ for all $x \in \Omega$ and so, $0$ is a viscosity supersolution to \eqref{pr:HJB-regularized}. This, \Cref{lem:comparison-principle}, and \eqref{ineq:monotonicity} prove, for all $0 < \varepsilon \leq 1/2$,
	\begin{align*}
		u \leq u_\varepsilon \leq 0 \text{ in } \Omega.
	\end{align*}
	Thus, the Alexandrov maximum principle \cite[Theorem 2.8]{Figalli2017} prove in two space dimensions that
	\begin{align*}
		(u(x) - u_\varepsilon(x))^2 \leq u(x)^2 \leq \mathrm{diam}(\Omega) (5C_2^2\varepsilon)^{1/16}\|f\|_{L^1(\Omega)}/(4\lambda)^{1/16}
	\end{align*}
	for all $x \in \overline{\Omega\setminus\Omega'}$.
	This and \eqref{ineq:proof-convergence-comparison-principle} conclude the proof of \eqref{ineq:rate-eps}.\qedhere
\end{proof}
\begin{remark}[inhomogeneous data on convex domain]
	Under the assumption that the Alexandrov solution $u$ in \Cref{thm:convergence} is strictly convex for given inhomogeneous data $g \neq 0$ on a merely convex domain, the convergence result from \Cref{thm:convergence}(a) remains valid. 
\end{remark}

The stability of Alexandrov solutions and \Cref{thm:convergence} provide
convergence for more general right-hand sides.
\begin{definition}[class of admissible right-hand sides]\label{def:L1plus}
We define the class 
\begin{align*}
 L^1_\uparrow(\Omega):=
  \{ f\in L^1(\Omega) :     
     \lim_{j\to\infty}\|f-f_j\|_{L^1(\Omega)}=0 \text{ for some sequence }
     f_j\in C(\overline \Omega) &
     \\
     \text{ with } 
      0<f_{j}\leq f_{j+1} \text{ in } \overline\Omega
      \text{ for all }j\in\mathbb N
  \}&
\end{align*}
of $L^1$ functions that can be represented as the limit of monotonically increasing
sequences of uniformly positive continuous functions.
\end{definition}
\begin{remark}[on the class $L^1_\uparrow(\Omega)$]
 The set $L^1_\uparrow(\Omega)$ can equivalently be defined as the set
 of monotone $L^1$ limits of positive smooth functions from
 $C^\infty(\overline\Omega)$.
 It is a strict subset of $L^1(\Omega)$. For example, the (shifted)
 indicator function of a nowhere dense set of positive Lebesgue
 measure is not contained in $L^1_\uparrow(\Omega)$.
 However, the positive continuous functions of $C(\overline\Omega)$
 as well as positive piecewise polynomial functions are included.
\end{remark}
\begin{theorem}[convergence for non-smooth right-hand side]\label{thm:convergence-nonsmooth}
	Let \Cref{assumption:structure} hold and let
	$f \in L^1_\uparrow(\Omega)$ and $g \in H^2(\Omega) \cap C^{1,\beta}(\overline{\Omega})$ with $0 < \beta < 1$
	be given.
	For any monotonically decreasing sequence $\varepsilon_j \searrow 0$,
	the sequence $(u_{\varepsilon_j})_{j \in \mathbb{N}}$ of 
	strong solutions $u_{\varepsilon_j} \in H^2(\Omega)$ to 
	\eqref{pr:HJB-regularized} with $\varepsilon \coloneqq \varepsilon_j$
	converges, as $j \to \infty$, uniformly in $\Omega$ to the Alexandrov solution 
	$u$ to the Monge--Amp\`ere equation \eqref{pr:Monge-Ampere}.
\end{theorem}
\begin{proof}
	Given $f \in L^1_\uparrow(\Omega)$, let the sequence 
	$(f_k)_{k\in \mathbb{N}} \subset C^\infty(\overline{\Omega})$ approximate
	$f$ in $L^1(\Omega)$ from below,
	i.e., $0<\lambda\leq f_{k} \leq f_{k+1} \leq f$ a.e.\ in
	$\Omega$ for all $k \in \mathbb{N}$ and $\lim_{k \to \infty} \|f - f_k\|_{L^1(\Omega)} = 0$.
	The comparison principle \cite[Theorem 2.10]{Figalli2017}
	proves that the sequence of Alexandrov solutions $u_k \in C(\overline{\Omega})$ to
	\eqref{pr:Monge-Ampere-HJB} with $f$ replaced by $f_k$
	satisfies the monotonicity $u \leq u_{k+1} \leq u_k$ pointwise in $\overline{\Omega}$ for all $k \in \mathbb{N}$.
	The stability of Alexandrov solutions \cite[Corollary 2.12, Proposition 2.16]{Figalli2017}
	proves that $(u_k)_{k \in \mathbb{N}}$ converges locally uniformly to $u$ as $k \to \infty$.
	This and the monotonicity of the sequence $(u_k)_{k \in \mathbb{N}}$ show the uniform convergence
	$\lim_{k \to \infty}\|u - u_k\|_{L^\infty(\Omega)} = 0$.
	For any pair $(j,k) \in \mathbb{N}^2$, let $u_{\varepsilon_j}^{(k)} \in H^2(\Omega)$ solve
	\begin{align*}
		F_{\varepsilon_j}(f_k;x,\D^2 u_{\varepsilon_j}^{(k)}(x)) = 0 \text{ in } \Omega \quad\text{and}\quad u_{\varepsilon_j}^{(k)} = g \text{ on } \partial \Omega.
	\end{align*}
	\Cref{lem:monotonicity} and \Cref{thm:convergence} imply that
	the sequence $(u_{\varepsilon_j}^{(k)})_{j \in \mathbb{N}}$ is pointwise monotonically
	decreasing with the uniform limit $u_k$ for $j \to \infty$, i.e.,
	\begin{align}
		u_k \leq u_{\varepsilon_{j+1}}^{(k)} \leq u_{\varepsilon_{j}}^{(k)} \text{ in } \overline{\Omega} \text{ for all } j,k \in \mathbb{N} 
		\quad\text{and}\quad \lim_{j \to \infty} \|u_k - u_{\varepsilon_j}^{(k)}\|_{L^\infty(\Omega)} = 0.
		\label{ineq:u_k-monotonicity}
	\end{align}
	In particular, there exists, for each $k \in \mathbb{N}$, an index $j(k) \in \mathbb{N}$ such that $\|u_k - u^{(k)}_{\varepsilon_{j}}\|_{L^\infty(\Omega)} \leq 1/k$ for all $j \geq j(k)$.
	The sequence $(j(k))_{k \in \mathbb{N}}$ can be chosen so that $j(k) < j(k+1)$ with $\lim_{k \to \infty} j(k) = \infty$.
	This and the triangle inequality provide
	\begin{align}
		\|u - u_{\varepsilon_{j(k)}}^{(k)}\|_{L^\infty(\Omega)} \leq \|u - u_k\|_{L^\infty(\Omega)} + \|u_k - u_{\varepsilon_{j(k)}}^{(k)}\|_{L^\infty(\Omega)} \to 0 \quad\text{as } k \to \infty.
		\label{ineq:uniform-convergence-subsequence-1}
	\end{align}
	Recall $u_{\varepsilon_j}^{(k)}\in C^2(\Omega)$ from
	\Cref{thm:existence-uniqueness-regularized}(d) for all pairs $(j,k)$.
	From $f_k\leq f_{k+1}$ in $\overline{\Omega}$, we deduce
	$$
	0=F_{\varepsilon_j}(f_k;x,\D^2 u_{\varepsilon_j}^{(k)} (x))
	\leq 
	F_{\varepsilon_j}(f_{k+1};x,\D^2 u_{\varepsilon_j}^{(k)} (x))
	\quad\text{in }\Omega,
	$$
	whence $u_{\varepsilon_j}^{(k)}$ is supersolution to
	$
	F_{\varepsilon_j}(f_{k+1};x,\D^2 u_{\varepsilon_j}^{(k+1)} (x)) = 0.
	$
	The comparison principle from \Cref{lem:comparison-principle} verifies
	$u_{\varepsilon_j}^{(k+1)} \leq u_{\varepsilon_j}^{(k)}$ in $\overline{\Omega}$ for any $j,k \in \mathbb{N}$.
	This and the limit $\|u_{\varepsilon_{j}} - u_{\varepsilon_j}^{(k)}\|_{L^\infty(\Omega)} \lesssim \|u_{\varepsilon_{j}} - u_{\varepsilon_j}^{(k)}\|_{H^2(\Omega)} \lesssim \|f - f_k\|_{L^1(\Omega)} \to 0$ as $k \to \infty$ from \Cref{thm:existence-uniqueness-regularized}(a) show
	that $u_{\varepsilon_j}$ is the pointwise limit of the sequence $(u_{\varepsilon_j}^{(k)})_{k \in \mathbb{N}}$ as $k \to \infty$ with
	\begin{align}
		u_{\varepsilon_j} \leq u_{\varepsilon_{j}}^{(k)} \text{ in } \overline{\Omega} \text{ for all } j,k \in \mathbb{N}.
		\label{ineq:monotonicity-u-eps-in-k}
	\end{align}
	Since $u_k \leq u_{\varepsilon_{j+1}}^{(k)} \leq u_{\varepsilon_j}^{(k)}$ holds pointwise in $\overline{\Omega}$ for any $j,k \in \mathbb{N}$ from \eqref{ineq:u_k-monotonicity}, 
	the passage to the limit as $k \to \infty$ results in $u \leq u_{\varepsilon_{j+1}} \leq u_{\varepsilon_{j}}$ in $\overline{\Omega}$.
	This, $u \leq u_{\varepsilon_{j(k)}} \leq u_{\varepsilon_{j(k)}}^{(k)}$ in $\overline{\Omega}$ from \eqref{ineq:monotonicity-u-eps-in-k} applied to the subsequence $(j(k))_{k\in \mathbb{N}}$, and \eqref{ineq:uniform-convergence-subsequence-1} confirm
	\begin{align}
		\|u - u_{\varepsilon_{j(k)}}\|_{L^\infty(\Omega)} \leq \|u - u^{(k)}_{\varepsilon_{j(k)}}\|_{L^\infty(\Omega)} \to 0 \text{ as } k \to \infty.
	\end{align}
	In other words, $u$ is the uniform limit of the subsequence $(u_{\varepsilon_{j(k)}})_{k \in \mathbb{N}}$.
	Hence, the monotonicity $u \leq u_{\varepsilon_{j+1}} \leq u_{\varepsilon_{j}}$ in $\overline{\Omega}$ for all $j \in \mathbb{N}$ concludes that the whole sequence $(u_{\varepsilon_j})_{j \in \mathbb{N}}$ converges uniformly to $u$.
\end{proof}

\section{Finite element discretization}\label{sec:discretization}
This section introduces the mixed formulation in 
\cite{GallistlSueli2019} for the approximation of the regularized PDE 
\eqref{pr:HJB-regularized} with finite elements.
The problem \eqref{pr:HJB-regularized} at hand is a special case of
\cite{SmearsSueli2014,GallistlSueli2019}, where $F_\varepsilon$ is 
independent of $u$ and $\nabla u$.
Throughout this section, we let $f \in L^1(\Omega)$ with $f \geq 0$
and $g \in H^2(\Omega)$ and fix $0  < \varepsilon \leq 1/2$.

\subsection{Review on the mixed formulation}
Recall the space $W = H^1_t(\Omega;\R^2)$ of Sobolev vector fields
with weak gradient and vanishing tangential trace,
endowed with the norm $\|\D \cdot\|_{L^2(\Omega)}$. 
This is indeed a norm of $W$, cf.~\cite[Eq.~(2.11)]{GallistlSueli2019}. 
Since the domain $\Omega$ is convex, any $w \in W$ satisfies
the following generalized Miranda--Talenti inequality
\cite{CostabelDauge1999}
\begin{align}
	\|\D w\|_{L^2(\Omega)}^2 
	\leq \|\div w\|_{L^2(\Omega)}^2 + \|\rot w\|_{L^2(\Omega)}^2.
	\label{ineq:norm-W}
\end{align}
On the other hand, in two space dimensions, the elementary bound
\begin{align}
	\|\div w\|_{L^2(\Omega)}^2 
	+ \|\rot w\|_{L^2(\Omega)}^2
	 \leq 2\|\D w\|_{L^2(\Omega)}^2
	\label{ineq:norm-W-reverse}
\end{align}
holds.
The substitution 
$w_\varepsilon \coloneqq \nabla (u_\varepsilon - g) \in W$
for the strong solution $u_\varepsilon \in H^2(\Omega)$ 
to \eqref{pr:HJB-regularized} leads to 
$F_\varepsilon(f;x,\D (w_\varepsilon + \nabla g)(x)) = 0$ a.e.~in $\Omega$. 
This and the surjective testing with functions in $\div W$ 
\cite{GallistlSueli2019} motivates the definition of the semilinear 
form $a=a_\varepsilon: W \times W \to \R$ 
by, for an open subset $\omega  \subset \Omega$ and all $w,z \in W$,
\begin{align*}
	a(w,z) 
	&\coloneqq -\int_\Omega 
	  F_{\gamma,\varepsilon}(f;\cdot, \D (w + \nabla g)) \div z \d{x} 
	  + \sigma\int_\Omega \rot w\, \rot z \d{x}
\end{align*}
with $F_{\gamma,\varepsilon}$ from \eqref{def:F-gamma-eps} 
and the stabilization parameter
$\sigma \coloneqq 1 - \sqrt{1 - \delta(\varepsilon)}/2$. 
The mixed problem seeks $w_\varepsilon \in W$ 
and 
$
 u_\varepsilon \in \mathcal{A} 
  \coloneqq g + V$ with $V = H^1_0(\Omega)
$ 
such that
\begin{align}
	a(w_\varepsilon,z) &= 0 \quad\text{for all } z \in W,\label{pr:mixed-1}\\
	(\nabla u_\varepsilon, \nabla v)_{L^2(\Omega)} &= (w_\varepsilon + \nabla g, \nabla v)_{L^2(\Omega)} \quad\text{for all } v \in V. \label{pr:mixed-2}
\end{align}
\begin{theorem}[solution to mixed problem]\label{thm:existence-uniqueness-mixed}
	There exits a unique solution 
	$(w_\varepsilon, u_\varepsilon) \in W \times \mathcal{A}$ 
	to \eqref{pr:mixed-1}--\eqref{pr:mixed-2} with 
	$w_\varepsilon = \nabla (u_\varepsilon - g)$. 
	The function $u_\varepsilon \in H^2(\Omega)$ is strong solution 
	to \eqref{pr:HJB-regularized}.
\end{theorem}

The proof is essentially contained in
 \cite{SmearsSueli2014,GallistlSueli2019} and sketched below
 for possibly inhomogeneous boundary data.
 It is based on the following lemma.

\begin{lemma}[monotonicity and Lipschitz continuity of $a$]\label{lem:monotonicity-Lipschitz-continuity}
	There exist positive constants
	 $c_\mathrm{mon}\approx \varepsilon$ and 
	 $C_\mathrm{Lip}\approx 1$ such that any $w,w',z \in W$ satisfy (a)--(b).
	\begin{enumerate}
		\item[(a)] (monotonicity) $c_\mathrm{mon}\|\D(w - z)\|_{L^2(\Omega)}^2 \leq a(w,w-z) - a(z,w-z)$.
		\item[(b)] (Lipschitz continuity) 
		$a(w,z) - a(w',z) 
		\leq C_\mathrm{Lip}\|\D(w - w')\|_{L^2(\Omega)}
		\|\D z\|_{L^2(\Omega)}$.
	\end{enumerate}
\end{lemma}

\begin{proof}[Proof of \Cref{lem:monotonicity-Lipschitz-continuity}(a)]
	Abbreviate $\delta_w \coloneqq w - z$. The replacements $u_\varepsilon \coloneqq \D (w + \nabla g)$ and $v_\varepsilon \coloneqq \D(z - \nabla g)$ in \eqref{ineq:pw-estimate} lead to
	\begin{align*}
		|F_{\gamma,\varepsilon}(f; \cdot, \D (z + \nabla g)) - F_{\gamma,\varepsilon}(f; \cdot, \D (w + \nabla g)) - \div \delta_w| \leq \sqrt{1 - \delta(\varepsilon)}|\D \delta_w|
	\end{align*}
	a.e.~in $\Omega$ with $0 < \delta(\varepsilon) \leq 1$ from
	\Cref{sec:analysis-regularized-PDE}.
	Hence, it follows
	\begin{align*}
		&\|\div \delta_w\|^2_{L^2(\Omega)} + \sigma\|\rot\delta_w\|^2_{L^2(\Omega)}\\
		&\quad - \sqrt{1 - \delta(\varepsilon)}\|\D \delta_w\|_{L^2(\Omega)}\|\div \delta_w\|_{L^2(\Omega)}
		\leq a(w, \delta_w) - a(z, \delta_w).
	\end{align*}	
	This, the Young inequality 
	$\alpha\beta\sqrt{1-\delta(\varepsilon)} \leq \alpha^2\sqrt{1-\delta(\varepsilon)}/2 
	+ \beta^2\sqrt{1-\delta(\varepsilon)}/2$ for all $\alpha,\beta \geq 0$, 
	and \eqref{ineq:norm-W} conclude the proof of (a) 
	with $c_\mathrm{mon} \coloneqq 1 - \sqrt{1 - \delta(\varepsilon)}$.\\[1em]
	\emph{Proof of \Cref{lem:monotonicity-Lipschitz-continuity}(b).}
	Abbreviate $\delta_w \coloneqq w - w'$. 
	The choice $M \coloneqq \D(w + \nabla g)$ 
	and $N \coloneqq \D(w' + \nabla g)$ 
	in \eqref{ineq:F-eps-upper-bound} and the Cauchy inequality lead to
	\begin{align*}
		|a(w,z) - a(w',z)| 
		\leq \|\D \delta_w\|_{L^2(\omega)}\|\div z\|_{L^2(\omega)} 
		+ \sigma\|\rot \delta_w\|_{L^2(\omega)}\|\rot z\|_{L^2(\omega)}.
	\end{align*}
	The Cauchy inequality in $\R^2$ and \eqref{ineq:norm-W-reverse} then establish (b).
\end{proof}

\begin{proof}[Proof of \Cref{thm:existence-uniqueness-mixed}]
	\Cref{lem:monotonicity-Lipschitz-continuity} implies that the linear functional $A: W \to W^*$, $A w \coloneqq a(w,\cdot)$ is strongly monotone and Lipschitz continuous. Thus, the Browder-Minty theorem \cite[Theorem 10.49]{RenardyRogers2004} shows the existence of a unique solution $w_\varepsilon \in W$ to \eqref{pr:mixed-1}. The second equation \eqref{pr:mixed-1} is equivalent to the minimization of the strongly convex energy $E(v) \coloneqq \|\nabla(v - g) - w_\varepsilon\|_{L^2(\Omega)}^2$ in $\mathcal{A}$. Hence, the solution $u_\varepsilon \in \mathcal{A}$ to \eqref{pr:mixed-2} exists and is unique. Notice that the strong solution $\widetilde{u}_\varepsilon \in \mathcal{A} \cap H^2(\Omega)$ to \eqref{pr:HJB-regularized} and $\widetilde{w}_\varepsilon \coloneqq \nabla (\widetilde{u}_\varepsilon - g) \in W$ solve \eqref{pr:mixed-1}--\eqref{pr:mixed-2}. The uniqueness of $u_\varepsilon$ and $w_\varepsilon$ conclude $u_\varepsilon = \widetilde{u}_\varepsilon$ and $w_\varepsilon = \widetilde{w}_\varepsilon$.
\end{proof}

\subsection{Discretization}
Let $W_h \subset W$ and $V_h \subset V$ be closed linear subspaces 
of $W$ and $V$. The discrete mixed problem seeks 
$w^{(\varepsilon)}_h=w_h \in W_h$ 
and $u^{(\varepsilon)}_h = u_h \in \mathcal{A}_h \coloneqq g + V_h$ such that
\begin{align}
	a(w_h,z_h) &= 0 \quad\text{for all } z_h \in W_h,
	\label{pr:discrete-mixed-1}\\
	(\nabla u_h, \nabla v_h)_{L^2(\Omega)} &= (w_h + \nabla g, \nabla v_h)_{L^2(\Omega)} \quad\text{for all } v_h \in V_h. \label{pr:discrete-mixed-2}
\end{align}
The strong monotonicity and Lipschitz continuity of $a$ give rise
the following a~priori error estimate.
\begin{theorem}[a~priori]\label{thm:apriori}
	There exists a unique discrete solution $(w_h,u_h) \in W_h \times \mathcal{A}_h$ to \eqref{pr:discrete-mixed-1}--\eqref{pr:discrete-mixed-2} with the a~priori estimates
	\begin{align}
		\|\D(w_\varepsilon - w_h)\|_{L^2(\Omega)} &\leq c_{\mathrm{mon}}^{-1}C_\mathrm{Lip} \min_{z_h \in W_h} \|\D(w_\varepsilon - z_h)\|_{L^2(\Omega)},
		\label{ineq:a-priori-1}\\
		\|\nabla(u_\varepsilon - u_h)\|_{L^2(\Omega)} &\leq 2\|w_\varepsilon - w_h\|_{L^2(\Omega)} + \min_{v_h \in \mathcal{A}_h}\|\nabla (u_\varepsilon - v_h)\|_{L^2(\Omega)}.
		\label{ineq:a-priori-2}
	\end{align}
	Notice that $\|w_\varepsilon - w_h\|_{L^2(\Omega)}$ 
	in \eqref{ineq:a-priori-2} can be controlled by 
	$\|\D(w_\varepsilon - w_h)\|_{L^2(\Omega)}$ due the Friedrichs inequality.
\end{theorem}
\begin{proof}
	The existence and uniqueness of discrete solutions to \eqref{pr:discrete-mixed-1}--\eqref{pr:discrete-mixed-2} follow directly from the Browder-Minty theorem. 
	We abbreviate $e_w \coloneqq w_\varepsilon - w_h$ and $e_u \coloneqq u_\varepsilon - u_h$.
	The monotonicity of $a$ from \Cref{lem:monotonicity-Lipschitz-continuity}(a), $a(w_\varepsilon,z_h) = a(w_h,z_h) = 0$ for all $z_h \in W_h$, and the Lipschitz continuity of $a$ in \Cref{lem:monotonicity-Lipschitz-continuity}(b) imply
	\begin{align}
		c_\mathrm{mon}\|\D e_w\|_{L^2(\Omega)}^2 &\leq a(w_\varepsilon,e_w) - a(w_h,e_w) = a(w_\varepsilon, w_\varepsilon - z_h) - a(w_h,w_\varepsilon - z_h)\nonumber\\
		&\leq C_\mathrm{Lip}\|\D e_w\|_{L^2(\Omega)}\|w_\varepsilon - z_h\|_{L^2(\Omega)}.
		\label{ineq:proof-a-posteriori-a}
	\end{align}
	This shows \eqref{ineq:a-priori-1}.
	The equations \eqref{pr:mixed-2} and \eqref{pr:discrete-mixed-2} prove, for all $\phi_h \in V_h$, that $(\nabla e_u,\nabla \phi_h)_{L^2(\Omega)} = (e_w, \nabla \phi_h)_{L^2(\Omega)}$. Hence, any $v_h \in \mathcal{A}_h$ satisfies
	\begin{align*}
		\|\nabla e_u\|_{L^2(\Omega)}^2 = (\nabla e_u, \nabla(u_\varepsilon - v_h))_{L^2(\Omega)} + (e_w,\nabla(v_h - u_\varepsilon + e_u))_{L^2(\Omega)}.
	\end{align*}
	This, the Cauchy inequality, and the choice $v_h \coloneqq \arg \min_{\phi_h \in \mathcal{A}_h} \|\nabla(u_\varepsilon - \phi_h)\|_{L^2(\Omega)}$ conclude \eqref{ineq:a-priori-2}.
\end{proof}

We deduce the following finite element convergence result
for the Monge--Amp\`ere equation.
In order to keep the technical details to a minimum, we assume $\Omega$ to be polygonal and, according to \Cref{assumption:structure}, $g \equiv 0$.
The extension to strictly convex domains is however possible; either with the techniques from \cite{Gallistl2019} or with other finite element approaches for curved domains, cf.~\cite{NEILAN2020105} and the references therein.

Let a quasi-uniform sequence of regular triangulations $(\mathcal{T}_\ell)_{\ell \in \mathbb{N}}$ 
be given such that $\mathcal{T}_\ell$ covers the domain $\cup_{T \in \mathcal{T}_\ell} T = \overline{\Omega}$ for all $\ell \in \mathbb{N}$ 
and the maximal mesh-size $h_\ell$ of $\mathcal{T}_\ell$ vanishes in the limit $\lim_{\ell \to \infty} h_\ell = 0$.
We discretize the variables $u$ and $w$ with Lagrange finite elements of order $k$; $V_\ell \coloneqq S^k_0(\mathcal{T}_\ell) \subset V$ and $W_\ell \coloneqq S^k_t(\mathcal{T}_\ell;\R^2) \subset W$.
For any $0 < \varepsilon \leq 1/2$ and $\ell \in \mathbb{N}$, let $(u_\ell^{(\varepsilon)}, w_\ell^{(\varepsilon)}) \in V_\ell \times W_\ell$ denote the solution to the discrete problem \eqref{pr:discrete-mixed-1}--\eqref{pr:discrete-mixed-2} with $V_h \coloneqq V_\ell$ and $W_h \coloneqq W_\ell$.

\begin{corollary}[convergence of FEM]\label{cor:convergence-FEM}
	Let \Cref{assumption:structure} hold and let
	$f \in L^1_\uparrow(\Omega)$ and $g \in H^2(\Omega) \cap C^{1,\beta}(\overline{\Omega})$ with $0 < \beta < 1$
	be given.
	The discrete solutions $u_\ell^{(\varepsilon)} \in \mathcal{A}_\ell$
 	converge uniformly to the Alexandrov solution $u \in C(\overline{\Omega})$ to the Monge--Amp\`ere equation \eqref{pr:Monge-Ampere}
 	in the following sense:
 	Given $\nu > 0$,
	there exist $\varepsilon_0 > 0$ such that
	\begin{align}
		\|u - u^{(\varepsilon)}_\ell\|_{L^\infty(\Omega)}
		\leq \nu
		\quad\text{for all } \varepsilon \leq \varepsilon_0
		\text{ and } \ell \geq \ell_0(\varepsilon)
		\label{ineq:convergence-FEM}
	\end{align}
	with a level $\ell_0(\varepsilon) \in \mathbb{N}$ that depends on the regularization parameter $\varepsilon$.
\end{corollary}

\begin{proof}
	The first part of the proof establishes an upper bound for $\|u - u_\ell^{(\varepsilon)}\|_{L^\infty(\Omega)}$ for fixed $0 < \varepsilon \leq 1/2$ and $\ell \in \mathbb{N}$.
	Let $v_\ell^{(\varepsilon)} \in V \cap H^2(\Omega)$ denote the solution to the Poisson problem $\Delta v_\ell^{(\varepsilon)} = \div w_\ell^{(\varepsilon)}$ in $\Omega$. The composition of the interpolation $\I_\ell$ onto the Morley finite element space \cite[Section 2.2]{Gallistl2015} and the conforming companion $\mathrm{J}_\ell$ from \cite[Poposition 2.6]{Gallistl2015} leads to $\psi_\ell^{(\varepsilon)} \coloneqq \mathrm{J}_\ell \I_\ell v_\ell^{(\varepsilon)} \in P_6(\mathcal{T}_\ell) \cap V \cap H^2(\Omega)$ such that
	\begin{align}
		h_\ell^{-2}\|v_\ell^{(\varepsilon)} - \psi_\ell^{(\varepsilon)}\|_{L^2(\Omega)} + \|\D^2(v_\ell^{(\varepsilon)} - \psi_\ell^{(\varepsilon)})\|_{L^2(\Omega)} \lesssim \|(1 - \Pi_\ell)\D^2 v_\ell^{(\varepsilon)}\|_{L^2(\Omega)}
		\label{ineq:proof-convergence-FEM-best-approximation}
	\end{align}
	with the $L^2$ orthogonal projection $\Pi_\ell$ onto the space $P_0(\mathcal{T}_\ell;\M)$ of matrix-valued piecewise constants. The triangle inequality shows
	\begin{align}
		\|u - u_\ell^{(\varepsilon)}\|_{L^\infty(\Omega)} &\leq \|u - u_\varepsilon\|_{L^\infty(\Omega)} + \|u_\varepsilon - v_\ell^{(\varepsilon)}\|_{L^\infty(\Omega)}\nonumber\\
		&\qquad + \|v_\ell^{(\varepsilon)} - \psi_\ell^{(\varepsilon)}\|_{L^\infty(\Omega)} + \|\psi_\ell^{(\varepsilon)} - u_\ell^{(\varepsilon)}\|_{L^\infty(\Omega)}.
		\label{ineq:proof-convergence-FEM-split}
	\end{align}
	The Sobolev embedding $H^2(\Omega) \subset C^{0,\vartheta}(\overline{\Omega})$ for any $0 < \vartheta < 1$ \cite[Theorem 4.12 (II)]{AdamsFournier2003} and the Miranda-Talenti estimate \cite[Lemma 2]{SmearsSueli2013} yield
	\begin{align*}
		\|u_\varepsilon - v_\ell^{(\varepsilon)}\|_{L^\infty(\Omega)} \lesssim \|\D^2(u_\varepsilon - v_\ell^{(\varepsilon)})\|_{L^2(\Omega)} \leq \|\Delta(u_\varepsilon - v_\ell^{(\varepsilon)})\|_{L^2(\Omega)}.
	\end{align*}
	This, $\Delta (u_\varepsilon - v_\ell^{(\varepsilon)}) = \div (w_\varepsilon - w_\ell^{(\varepsilon)})$, and the a~priori estimate \eqref{ineq:a-priori-1} prove
	\begin{align}
		\|u_\varepsilon - v_\ell^{(\varepsilon)}\|_{L^\infty(\Omega)} \lesssim \|\D^2(u_\varepsilon - v_\ell^{(\varepsilon)})\|_{L^2(\Omega)} \lesssim \varepsilon^{-1} \min_{z_\ell \in W_\ell} \|\D(w_\varepsilon - z_\ell)\|_{L^2(\Omega)}.
		\label{ineq:proof-convergence-FEM-term-1}
	\end{align}
	The Sobolev embedding $H^2(\Omega) \subset C^{0,\vartheta}(\overline{\Omega})$ and \eqref{ineq:proof-convergence-FEM-best-approximation} verify
	\begin{align*}
		\|v_\ell^{(\varepsilon)} - \psi_\ell^{(\varepsilon)}\|_{L^\infty(\Omega)} \lesssim \|\D^2 (v_\ell^{(\varepsilon)} - \psi_\ell^{(\varepsilon)})\|_{L^2(\Omega)} \lesssim \|(1 - \Pi_\ell) \D^2 v_\ell^{(\varepsilon)}\|_{L^2(\Omega)}.
	\end{align*}
	This, the best-approximation property of $\Pi_\ell$, the triangle inequality, and \eqref{ineq:proof-convergence-FEM-term-1} show
	\begin{align}
		\|v_\ell^{(\varepsilon)} - \psi_\ell^{(\varepsilon)}\|_{L^\infty(\Omega)} &\lesssim  \|(1 - \Pi_\ell) \D^2 v_\ell^{(\varepsilon)}\|_{L^2(\Omega)} \leq \|\D^2 v_\ell^{(\varepsilon)} - \Pi_\ell \D^2 u_\varepsilon\|_{L^2(\Omega)}\nonumber\\
		& \lesssim \|(1 - \Pi_\ell) \D^2 u_\varepsilon\|_{L^2(\Omega)} + \varepsilon^{-1}\min_{z_\ell \in W_\ell} \|\D(w_\varepsilon - z_\ell)\|_{L^2(\Omega)}.
		\label{ineq:proof-convergence-FEM-term-2}
	\end{align}
	The function $\psi_\ell^{(\varepsilon)} - u_\ell^{(\varepsilon)}$ is discrete and so, a scaling argument leads to 
        $\|\psi_\ell^{(\varepsilon)} - u_\ell^{(\varepsilon)}\|_{L^\infty(\Omega)} \lesssim h_\ell^{-1}\|\psi_\ell^{(\varepsilon)} - u_\ell^{(\varepsilon)}\|_{L^2(\Omega)}$. 
        Thus, the triangle inequality and \eqref{ineq:proof-convergence-FEM-best-approximation} confirm
	\begin{align}
		&\|\psi_\ell^{(\varepsilon)} - u_\ell^{(\varepsilon)}\|_{L^\infty(\Omega)}\leq h_\ell\|(1 - \Pi_\ell) \D^2 v_\ell^{(\varepsilon)}\|_{L^2(\Omega)} + h_\ell^{-1}\|v_\ell^{(\varepsilon)} - u_\ell^{(\varepsilon)}\|_{L^2(\Omega)}.
		\label{ineq:proof-convergence-FEM-term-3}
	\end{align}
	Notice that $u_\ell^{(\varepsilon)} \in S^k_0(\mathcal{T}_\ell)$ is the finite element approximation to the Poisson problem $\Delta v_\ell^{(\varepsilon)} = \div w_\ell^{(\varepsilon)}$ in the convex domain $\Omega$ and so, standard a~priori analysis proves $\|v_\ell^{(\varepsilon)} - u_\ell^{(\varepsilon)}\|_{L^2(\Omega)} \lesssim h_\ell^2\|\div w_\ell^{(\varepsilon)}\|_{L^2(\Omega)}$. Hence, the triangle inequality and \eqref{ineq:a-priori-1} show $\|v_\ell^{(\varepsilon)} - u_\ell^{(\varepsilon)}\|_{L^2(\Omega)} \lesssim \varepsilon^{-1} h_\ell^2\min_{z_\ell \in W_\ell} \|\D(w_\varepsilon - z_\ell)\|_{L^2(\Omega)} + h_\ell^2\|\div w_\varepsilon\|$.
	The combination of this with \eqref{ineq:proof-convergence-FEM-split}--\eqref{ineq:proof-convergence-FEM-term-3} and $h_\ell \lesssim 1$ result in
	\begin{align}
		\|u - u_\ell^{(\varepsilon)}\|_{L^\infty(\Omega)} &\lesssim \|u - u_\varepsilon\|_{L^\infty(\Omega)} + \|(1 - \Pi_\ell)\D^2 u_\varepsilon\|_{L^2(\Omega)}\nonumber\\
		&\qquad + \varepsilon^{-1}\min_{z_\ell \in W_\ell} \|\D(w_\varepsilon - z_\ell)\|_{L^2(\Omega)} + h_\ell\|\div w_\varepsilon\|_{L^2(\Omega)}.
		\label{ineq:proof-convergence-FEM-upper-bpund}
	\end{align}
	Given $nu > 0$, \eqref{ineq:convergence-FEM} follows immediately from \eqref{ineq:proof-convergence-FEM-upper-bpund}, the convergence result from \Cref{thm:convergence-nonsmooth}, 
	and the density of $P_0(\mathcal{T}_\ell;\M)$ in $L^2(\Omega;\M)$ and of $W_\ell$ in $W$ as $\ell \to \infty$.
\end{proof}

\section{Numerical experiments}\label{sec:numerical-experiments}

In this section we present four numerical examples.

\subsection{Setup}

In all examples, the domain is the unit square
$\Omega=(0,1)^2$.
The variables $u$ and $w$ are both discretized with first-order
Lagrange finite elements
on a sequence of uniformly refined triangular meshes
with mesh-sizes varying from $h=2^0$
to $h=2^{-10}$;
the mesh size $h$ refers to the maximal diameter of the 
elements in the triangulation.
The coarsest mesh consists of the four triangles
that emerge from connecting the corners of the square
with the midpoint $(1/2,1/2)$ by straight lines.
The discrete nonlinear problem is solved with
the semismooth Newton method proposed in \cite{SmearsSueli2014}.
The regularization parameters are chosen as
$\varepsilon=10^{-1},10^{-2},10^{-3}$.
The diagrams display the errors between $u$ and $u_h$
in the global
$L^\infty$ norm, the $L^2$ norm, and the $H^1$ seminorm
($L^2$ norm of the gradient) versus
to $1/h$.

\subsection{First experiment}

In this example from \cite{DeanGlowinski2006},
the exact solution is given by
$$
 u(x) = \frac{(2|x|)^{3/2}}{3}
$$
with $f(x) = 1/|x|$.
The solution belongs to $H^{5/2-\nu}(\Omega)$
for any $\nu>0$, but not to $C^2(\overline{\Omega})$.
The convergence history is displayed in
Figure~\ref{f:conv1}. 
The $H^1$ error converges at order $h$, the $L^2$ error
at order $h^2$, and the $L^\infty$ error at an order
between $h$ and $h^2$. The asymptotic behaviour starts
from the first refinement.
The error curves show no differences
with respect to the regularization parameter $\varepsilon$.
This observation can be explained as follows.
In this case with $u\in C^2(\Omega)$, it can be explicitly computed
that
\begin{align}\label{ineq:effect-of-regularization-pointwise}
	|D^2 u(x)|^2
	\leq 
	\frac{(1-\varepsilon)^2+\varepsilon^2}{\varepsilon(1-\varepsilon)} f(x)
	\qquad\text{for any }x\in\Omega
\end{align}
holds for any $0<\varepsilon\leq 1/3$.
We apply \Cref{lem:regularization} pointwise and conclude that 
the Monge--Amp\`ere solution $u$ is the viscosity solution
to the regularized problem \eqref{pr:HJB-regularized}
for all $0<\varepsilon\leq 1/3$.
Therefore, we expect that our choices of
$\varepsilon=10^{-1},10^{-2},10^{-3}$
do not lead to plateaus in the convergence history.


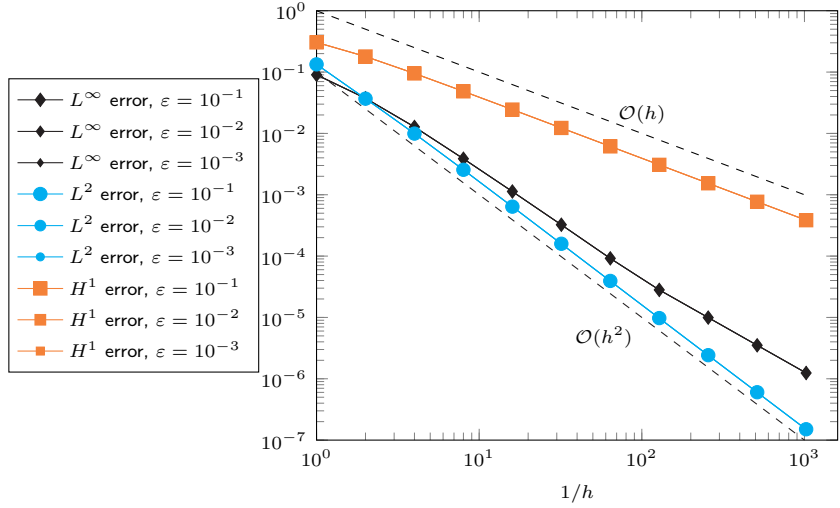
\begin{figure}
	\begin{tikzpicture}
		\begin{loglogaxis}[legend pos=south west,legend cell align=left,
			legend style={fill=none},
			ymin=1e-7,ymax=1e0,
			xmin=1e0,xmax=1.6e3,
			ytick={1e-7,1e-6,1e-5,1e-4,1e-3,1e-2,1e-1,1e0},
			xtick={1e0,1e1,1e2,1e3,1e4}]
			\pgfplotsset{
				cycle list={%
					{Black, mark=diamond*, mark size=2.5pt},
					{Black, mark=diamond*, mark size=2pt},
					{Black, mark=diamond*, mark size=1.5pt},
					{Cyan, mark=*, mark size=2.5pt},
					{Cyan, mark=*, mark size=2pt},
					{Cyan, mark=*, mark size=1.5pt},
					{Orange, mark=square*, mark size=2.5pt},
					{Orange, mark=square*, mark size=2pt},
					{Orange, mark=square*, mark size=1.5pt},
				},
				legend style={
					at={(-0.11,.5)}, anchor=east},
				font=\sffamily\scriptsize,
				xlabel=$1/h$,ylabel=
			}
			\addplot+ table[x=hinv,y=Linferr]{convhist_ex1_epsi0.1.dat};
			\addlegendentry{$L^\infty$ error, $\varepsilon=10^{-1}$}
			\addplot+ table[x=hinv,y=Linferr]{convhist_ex1_epsi0.01.dat};
			\addlegendentry{$L^\infty$ error, $\varepsilon=10^{-2}$}
			\addplot+ table[x=hinv,y=Linferr]{convhist_ex1_epsi0.001.dat};
			\addlegendentry{$L^\infty$ error, $\varepsilon=10^{-3}$}
			\addplot+ table[x=hinv,y=L2err]{convhist_ex1_epsi0.1.dat};
			\addlegendentry{$L^2$ error, $\varepsilon=10^{-1}$}
			\addplot+ table[x=hinv,y=L2err]{convhist_ex1_epsi0.01.dat};
			\addlegendentry{$L^2$ error, $\varepsilon=10^{-2}$}
			\addplot+ table[x=hinv,y=L2err]{convhist_ex1_epsi0.001.dat};
			\addlegendentry{$L^2$ error, $\varepsilon=10^{-3}$}
			\addplot+ table[x=hinv,y=H1err]{convhist_ex1_epsi0.1.dat};
			\addlegendentry{$H^1$ error, $\varepsilon=10^{-1}$}
			\addplot+ table[x=hinv,y=H1err]{convhist_ex1_epsi0.01.dat};
			\addlegendentry{$H^1$ error, $\varepsilon=10^{-2}$}
			\addplot+ table[x=hinv,y=H1err]{convhist_ex1_epsi0.001.dat};
			\addlegendentry{$H^1$ error, $\varepsilon=10^{-3}$}
			\addplot+ [dashed,mark=none] coordinates{(1,1) (1e3,1e-3)};
			\node(z) at  (axis cs:100,1e-2)
			[above] {$\mathcal O (h)$};
			\addplot+ [dashed,mark=none] coordinates{(1,.1) (1e3,1e-7)};
			\node(z) at  (axis cs:100,1e-5)
			[below left] {$\mathcal O (h^{2})$};
		\end{loglogaxis}
	\end{tikzpicture}
	\caption{Convergence history for the first experiment.
		\label{f:conv1}
	}
\end{figure}

\subsection{Second experiment}

In this example from \cite{DeanGlowinski2006},
the exact solution
$$
u(x)=-\sqrt{2-|x|^2};
$$
to \eqref{pr:Monge-Ampere} with the right-hand side $f(x) = 2/(2-|x|^2)^2$ belongs to $H^{3/2-\nu}(\Omega)$
for any $\nu>0$, but not to
$H^{3/2}(\Omega)$.
The convergence history from
Figure~\ref{f:conv2} shows convergence of order $h^{1/2}$
in the $H^1$ and the $L^\infty$ norms and
the order $h^{3/2}$ in the $L^2$ norm.
While any choice of the regularization parameter $\varepsilon$
seems sufficient for the mentioned rate in $H^1$ or $L^\infty$
in the range of refinements, the $L^2$ error deteriorates
for the last three refinements when $\varepsilon=0.1$
is chosen. 

For the second example, \eqref{ineq:effect-of-regularization-pointwise} holds only on the subdomain $\Omega' \coloneqq \{x \in \Omega: |x| \leq 4/3\}$ for the parameter $\varepsilon = 0.1$. Therefore, we expect (at least asymptotically) plateaus for all error curves.
The rates are observable starting from the first
refinement.

\begin{figure}
   \begin{tikzpicture}
\begin{loglogaxis}[legend pos=south west,legend cell align=left,
                   legend style={fill=none},
                   ymin=1e-6,ymax=1e0,
                   xmin=1e0,xmax=1.6e3,
                   ytick={1e-6,1e-5,1e-4,1e-3,1e-2,1e-1,1e0},
                   xtick={1e0,1e1,1e2,1e3}]
\pgfplotsset{
cycle list={%
{Black, mark=diamond*, mark size=2.5pt},
{Black, mark=diamond*, mark size=2pt},
{Black, mark=diamond*, mark size=1.5pt},
{Cyan, mark=*, mark size=2.5pt},
{Cyan, mark=*, mark size=2pt},
{Cyan, mark=*, mark size=1.5pt},
{Orange, mark=square*, mark size=2.5pt},
{Orange, mark=square*, mark size=2pt},
{Orange, mark=square*, mark size=1.5pt},
},
legend style={
at={(-0.11,.5)}, anchor=east},
font=\sffamily\scriptsize,
xlabel=$1/h$,ylabel=
}
\addplot+ table[x=hinv,y=Linferr]{convhist_ex2_epsi0.1.dat};
\addlegendentry{$L^\infty$ error, $\varepsilon=10^{-1}$}
\addplot+ table[x=hinv,y=Linferr]{convhist_ex2_epsi0.01.dat};
\addlegendentry{$L^\infty$ error, $\varepsilon=10^{-2}$}
\addplot+ table[x=hinv,y=Linferr]{convhist_ex2_epsi0.001.dat};
 \addlegendentry{$L^\infty$ error, $\varepsilon=10^{-3}$}
\addplot+ table[x=hinv,y=L2err]{convhist_ex2_epsi0.1.dat};
\addlegendentry{$L^2$ error, $\varepsilon=10^{-1}$}
\addplot+ table[x=hinv,y=L2err]{convhist_ex2_epsi0.01.dat};
\addlegendentry{$L^2$ error, $\varepsilon=10^{-2}$}
\addplot+ table[x=hinv,y=L2err]{convhist_ex2_epsi0.001.dat};
 \addlegendentry{$L^2$ error, $\varepsilon=10^{-3}$}
\addplot+ table[x=hinv,y=H1err]{convhist_ex2_epsi0.1.dat};
\addlegendentry{$H^1$ error, $\varepsilon=10^{-1}$}
\addplot+ table[x=hinv,y=H1err]{convhist_ex2_epsi0.01.dat};
\addlegendentry{$H^1$ error, $\varepsilon=10^{-2}$}
\addplot+ table[x=hinv,y=H1err]{convhist_ex2_epsi0.001.dat};
 \addlegendentry{$H^1$ error, $\varepsilon=10^{-3}$}
\addplot+ [dashed,mark=none] coordinates{(1,1) (1e4,1e-2)};
 \node(z) at  (axis cs:100,1e-1)
          [above] {$\mathcal O (h^{1/2})$};
\addplot+ [dashed,mark=none] coordinates{(1,1e-1) (1e4,1e-7)};
\node(z) at  (axis cs:10,1e-3)
          [below left] {$\mathcal O (h^{3/2})$};
\end{loglogaxis}
\end{tikzpicture}
   \caption{Convergence history for the second experiment.
            \label{f:conv2}
           }
\end{figure}
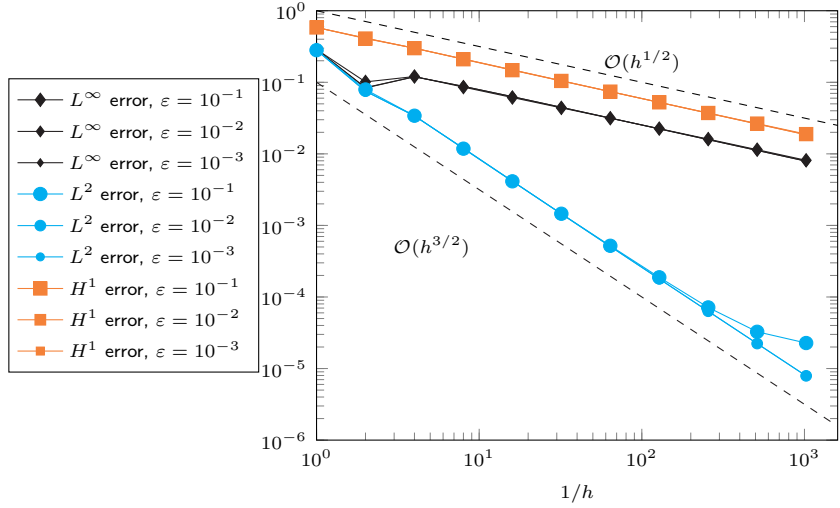

\subsection{Third experiment}

In this example, the exact solution is given by the
absolute value function
$$
u(x)=|x_1-1/2|;
$$
with $f(x) =0$.
The solution belongs to
$H^{3/2-\nu}(\Omega)$ for any $\nu>0$,
but not to $H^{3/2}(\Omega)$.
This example is not covered by
\Cref{cor:convergence-FEM}
because neither $f\in L^1_\uparrow(\Omega)$
nor $g\in H^2(\Omega)$.
The convergence history is displayed in
Figure~\ref{f:conv3}.
In this example, the effect of the regularization is
significant in the sense that the difference $u-u_\varepsilon$
dominates the discretization error on fine meshes,
which results in the error plateaus for small $h$.
For moderate choice of $h$ and $\varepsilon=0.001$,
the errors are of order $h^{1/4}$ in the 
$H^1$ and $L^\infty$ norms and of order
$h^{3/4}$ in the $L^2$ norm.
For any of the choice of $\varepsilon$, the error curve
shows flattening for fine meshes.
This provides a means of estimating the order of
convergence for the regularization error
$\|u-u_\varepsilon\|_{L^\infty(\Omega)}$.
Table~\ref{t:epsconv_ex3} shows the $L^\infty$
errors on the finest mesh for regularization
parameters $\varepsilon=2^{-j}$ for
$j\in\{1,\dots,11\}$ and the empirical convergence rate
with respect to $\varepsilon$.
For comparison, the error with respect to the penultimate
mesh is displayed as well. On meshes up to 
$h^{-9}$ the error stagnates with respect to mesh refinement;
this is the range where the empirical $\varepsilon$-rate
is meaningful. We observe values between $0.42$ and
$0.53$ for $\varepsilon$ between $2^{-4}$ and
$2^{-9}$.

It is worth mentioning that in this case
the exact solution belongs to the finite element space $V_h$
but does not coincide with the computed approximation
$u_h\in V_h$.
This is not only caused by the $\varepsilon$-regularization,
but also by the fact that the gradient $\nabla u$
is discontinuous, which results in boundary data
incompatible with the $H^1(\Omega)$ conforming finite element
space $W_h$.
This additional regularization is visible in the graph
of the computational solution in
Figure~\ref{f:surf3}.

\begin{figure}
   \begin{tikzpicture}
\begin{loglogaxis}[legend pos=south west,legend cell align=left,
                   legend style={fill=none},
                   ymin=1e-3,ymax=1e0,
                   xmin=1e0,xmax=1.6e3,
                   ytick={1e-6,1e-5,1e-4,1e-3,1e-2,1e-1,1e0},
                   xtick={1e0,1e1,1e2,1e3}]
\pgfplotsset{
cycle list={%
{Black, mark=diamond*, mark size=2.5pt},
{Black, mark=diamond*, mark size=2pt},
{Black, mark=diamond*, mark size=1.5pt},
{Cyan, mark=*, mark size=2.5pt},
{Cyan, mark=*, mark size=2pt},
{Cyan, mark=*, mark size=1.5pt},
{Orange, mark=square*, mark size=2.5pt},
{Orange, mark=square*, mark size=2pt},
{Orange, mark=square*, mark size=1.5pt},
},
legend style={
at={(-0.11,.5)}, anchor=east},
font=\sffamily\scriptsize,
xlabel=$1/h$,ylabel=
}
\addplot+ table[x=hinv,y=Linferr]{convhist_ex3_epsi0.1.dat};
\addlegendentry{$L^\infty$ error, $\varepsilon=10^{-1}$}
\addplot+ table[x=hinv,y=Linferr]{convhist_ex3_epsi0.01.dat};
\addlegendentry{$L^\infty$ error, $\varepsilon=10^{-2}$}
\addplot+ table[x=hinv,y=Linferr]{convhist_ex3_epsi0.001.dat};
 \addlegendentry{$L^\infty$ error, $\varepsilon=10^{-3}$}
\addplot+ table[x=hinv,y=L2err]{convhist_ex3_epsi0.1.dat};
\addlegendentry{$L^2$ error, $\varepsilon=10^{-1}$}
\addplot+ table[x=hinv,y=L2err]{convhist_ex3_epsi0.01.dat};
\addlegendentry{$L^2$ error, $\varepsilon=10^{-2}$}
\addplot+ table[x=hinv,y=L2err]{convhist_ex3_epsi0.001.dat};
 \addlegendentry{$L^2$ error, $\varepsilon=10^{-3}$}
\addplot+ table[x=hinv,y=H1err]{convhist_ex3_epsi0.1.dat};
\addlegendentry{$H^1$ error, $\varepsilon=10^{-1}$}
\addplot+ table[x=hinv,y=H1err]{convhist_ex3_epsi0.01.dat};
\addlegendentry{$H^1$ error, $\varepsilon=10^{-2}$}
\addplot+ table[x=hinv,y=H1err]{convhist_ex3_epsi0.001.dat};
 \addlegendentry{$H^1$ error, $\varepsilon=10^{-3}$}
\addplot+ [dashed,mark=none] coordinates{(1,4e-1) (1e4,4e-2)};
\node(z) at  (axis cs:100,.1)
          [below] {$\mathcal O (h^{1/4})$};
\addplot+ [dashed,mark=none] coordinates{(1,8e-2) (1e4,8e-5)};
\node(z) at  (axis cs:50,4e-3)
          [below left] {$\mathcal O (h^{3/4})$};
\end{loglogaxis}
\end{tikzpicture}
   \caption{Convergence history for the third experiment.
            \label{f:conv3}
           }
\end{figure}
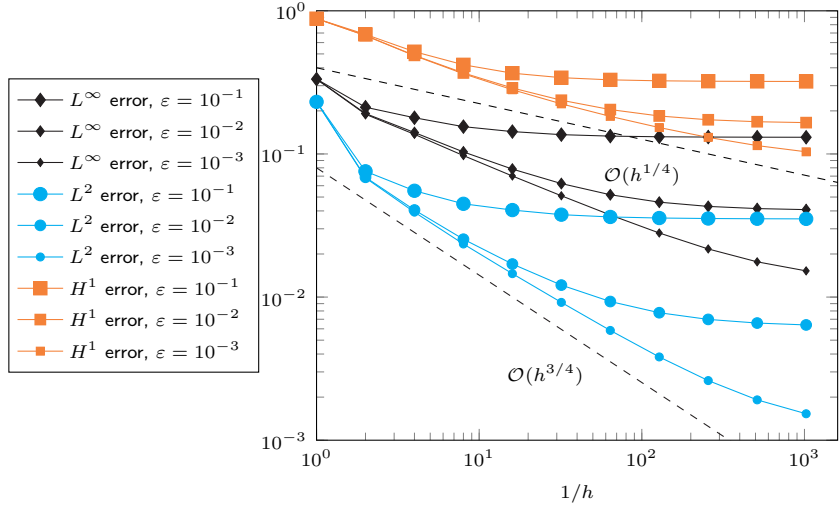

\begin{figure}
    \begin{center}
    \includegraphics[width=.7\textwidth]
               {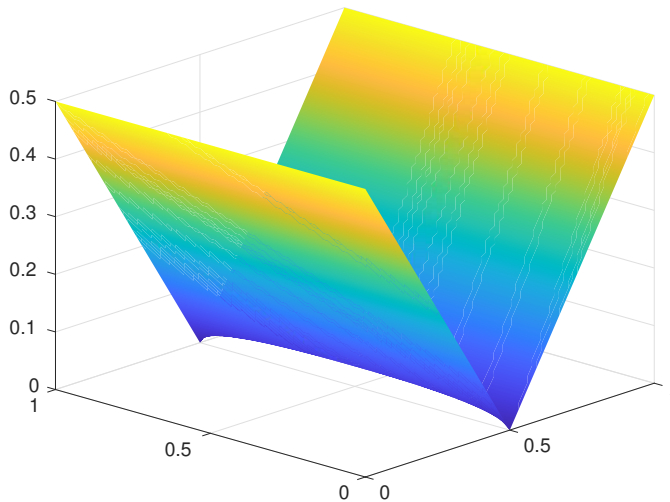}
    \end{center}
    \caption{Graph of the computed approximation in
             the third experiment;
             mesh size $h=1/64$ and $\varepsilon=10^{-3}$.
            \label{f:surf3}
           }
\end{figure}

\begin{table}                  
\begin{tabular}{c||c||c|c} 
$\varepsilon$  & $L^\infty$ error $h=2^{-9}$
               & $L^\infty$ error $h=2^{-10}$ & rate \\
\hline
$2^{-1}$  & 3.376e-01 & 3.376e-01 & --- \\
$2^{-2}$  & 2.193e-01 & 2.193e-01 & 0.62\\
$2^{-3}$  & 1.477e-01 & 1.476e-01 & 0.57\\
$2^{-4}$  & 1.022e-01 & 1.020e-01 & 0.53\\
$2^{-5}$  & 7.196e-02 & 7.162e-02 & 0.51\\
$2^{-6}$  & 5.127e-02 & 5.074e-02 & 0.49\\
$2^{-7}$  & 3.708e-02 & 3.627e-02 & 0.42\\
$2^{-8}$  & 2.751e-02 & 2.627e-02 & 0.48\\
$2^{-9}$  & 2.131e-02 & 1.951e-02 & 0.46\\
$2^{-10}$ & 1.753e-02 & 1.512e-02 & 0.36\\
$2^{-11}$ & 1.537e-02 & 1.245e-02 & 0.28
\end{tabular}
\caption{Convergence with respect to $\varepsilon$
         in the third experiment.
         \label{t:epsconv_ex3}
        }
\end{table}

\subsection{Fourth experiment}
In this example, the function
$$
	u(x,y) = -\left( \left(\sin(\pi x)\right)^{-1} +
	\left(\sin(\pi y)\right)^{-1} \right)^{-1}
$$
solves \eqref{pr:Monge-Ampere} with homogeneous Dirichlet boundary data and the right-hand side
$$
	f(x,y) = \frac{4\pi^2\sin(\pi x)^2\sin(\pi y)^2(2-\sin(\pi x)\sin(\pi y))}{(\sin(\pi x) + \sin(\pi y))^4}.
$$
The solution belongs to
$H^{2-\nu}(\Omega)$ for any $\nu>0$,
but not to $H^{2}(\Omega)$.
It can be computed explicitly that $u$ satisfies
$$
	|\D^2 u(x,y)|/f(x,y) \leq 80/\mathrm{dist}((x,y), \partial \Omega)^2 \text{ for all } (x,y) \in \Omega.
$$
Hence, \eqref{ineq:effect-of-regularization-pointwise} holds on the subdomain
$\Omega' \coloneqq \{x \in \Omega : \dist(x, \partial \Omega) \geq 10\sqrt{\varepsilon}\}$.
For homogeneous Dirichlet data, the arguments from \Cref{sec:convergence} lead to 
$$
	\|u - u_\varepsilon\|_{L^\infty(\Omega)} \leq \|u\|_{L^\infty(\Omega \setminus \Omega')} \leq 50\varepsilon^{1/4}.
$$
This bound has at least two consequences.
First, convergence of FEM is guaranteed
(under the assumption that the mesh-size is sufficient
small depending on the parameter $\varepsilon$) even
if this example is not covered by \Cref{cor:convergence-FEM}
(because $f$ attains values arbitrarily close to zero).
Second, we expect at least $1/4$ for the convergence rate
of $\|u - u_h^{(\varepsilon)}\|_{L^\infty(\Omega)}$ in $\varepsilon$
with sufficiently small $h$.

The convergence history is displayed in
Figure~\ref{f:conv4}.
Table~\ref{t:epsconv_ex4} shows the empirical convergence
with respect to $\varepsilon$.
On meshes with error stagnation with respect to $h$,
the convergence rate with respect to $\varepsilon$
is observed to be close to 1; this is better than the lower bound $1/4$ computed above.

\begin{figure}
\begin{tikzpicture}
\begin{loglogaxis}[legend pos=south west,legend cell align=left,
                   legend style={fill=none},
                   ymin=1e-5,ymax=1e1,
                   xmin=1e0,xmax=1.6e3,
                   ytick={1e-6,1e-5,1e-4,1e-3,1e-2,1e-1,1e0,1e1},
                   xtick={1e0,1e1,1e2,1e3}]
\pgfplotsset{
cycle list={%
{Black, mark=diamond*, mark size=2.5pt},
{Black, mark=diamond*, mark size=2pt},
{Black, mark=diamond*, mark size=1.5pt},
{Cyan, mark=*, mark size=2.5pt},
{Cyan, mark=*, mark size=2pt},
{Cyan, mark=*, mark size=1.5pt},
{Orange, mark=square*, mark size=2.5pt},
{Orange, mark=square*, mark size=2pt},
{Orange, mark=square*, mark size=1.5pt},
},
legend style={
at={(-0.11,.5)}, anchor=east},
font=\sffamily\scriptsize,
xlabel=$1/h$,ylabel=
}
\addplot+ table[x=hinv,y=Linferr]{convhist_ex4_epsi0.1.dat};
\addlegendentry{$L^\infty$ error, $\varepsilon=10^{-1}$}
\addplot+ table[x=hinv,y=Linferr]{convhist_ex4_epsi0.01.dat};
\addlegendentry{$L^\infty$ error, $\varepsilon=10^{-2}$}
\addplot+ table[x=hinv,y=Linferr]{convhist_ex4_epsi0.001.dat};
 \addlegendentry{$L^\infty$ error, $\varepsilon=10^{-3}$}
\addplot+ table[x=hinv,y=L2err]{convhist_ex4_epsi0.1.dat};
\addlegendentry{$L^2$ error, $\varepsilon=10^{-1}$}
\addplot+ table[x=hinv,y=L2err]{convhist_ex4_epsi0.01.dat};
\addlegendentry{$L^2$ error, $\varepsilon=10^{-2}$}
\addplot+ table[x=hinv,y=L2err]{convhist_ex4_epsi0.001.dat};
 \addlegendentry{$L^2$ error, $\varepsilon=10^{-3}$}
\addplot+ table[x=hinv,y=H1err]{convhist_ex4_epsi0.1.dat};
\addlegendentry{$H^1$ error, $\varepsilon=10^{-1}$}
\addplot+ table[x=hinv,y=H1err]{convhist_ex4_epsi0.01.dat};
\addlegendentry{$H^1$ error, $\varepsilon=10^{-2}$}
\addplot+ table[x=hinv,y=H1err]{convhist_ex4_epsi0.001.dat};
 \addlegendentry{$H^1$ error, $\varepsilon=10^{-3}$}
\addplot+ [dashed,mark=none] coordinates{(1,3e-2) (1e3,3e-5)};
\node(z) at  (axis cs:40,1e-3)
          [below left] {$\mathcal O (h)$};
\end{loglogaxis}
\end{tikzpicture}
   \caption{Convergence history for the fourth experiment.
            \label{f:conv4}
           }
\end{figure}
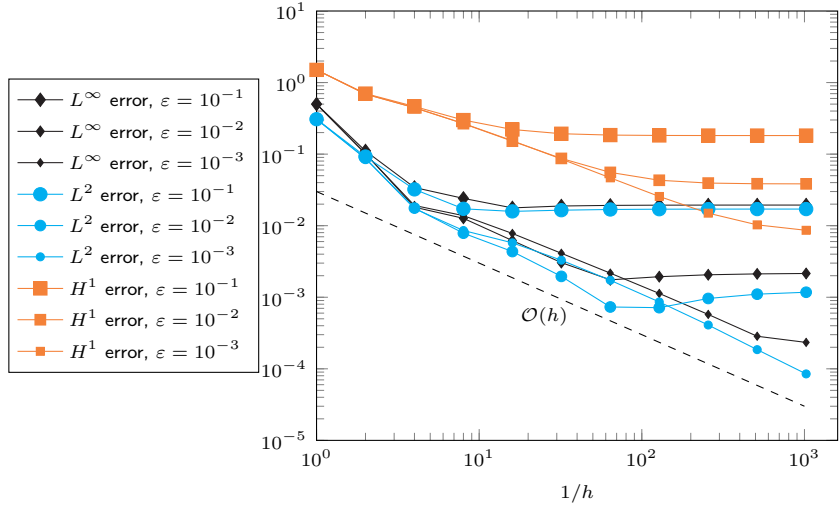

\begin{table}
	\begin{tabular}{c||c||c|c} 
		$\varepsilon$  & $L^\infty$ error $h=2^{-9}$
		& $L^\infty$ error $h=2^{-10}$ & rate \\
		\hline
		$2^{-1}$  &1.083e-01 & 1.083e-01 & ---\\
		$2^{-2}$  &4.869e-02 & 4.869e-02 & 1.15\\
		$2^{-3}$  &2.417e-02 & 2.417e-02 & 1.01\\
		$2^{-4}$  &1.228e-02 & 1.229e-02 & 0.97\\
		$2^{-5}$  &6.296e-03 & 6.316e-03 & 0.96\\
		$2^{-6}$  &3.246e-03 & 3.272e-03 & 0.94\\
		$2^{-7}$  &1.674e-03 & 1.702e-03 & 0.94\\
		$2^{-8}$  &8.554e-04 & 8.821e-04 & 0.94\\
		$2^{-9}$  &4.292e-04 & 4.525e-04 & 0.96\\
		$2^{-10}$ &2.861e-04 & 2.283e-04 & 0.98\\
		$2^{-11}$ &3.060e-04 & 1.527e-04 & 0.57
	\end{tabular}
	\caption{Convergence with respect to $\varepsilon$
		in the fourth experiment.
		\label{t:epsconv_ex4}
	}
\end{table}

\subsection{Conclusions from the computations}

In all examples, we do not observe deterioration of the
approximation constants for small choices of $\varepsilon$,
which cannot be excluded by our theory.
The convergence rate in $h$ (up to stagnation caused by
the regularization) is visible starting from the
coarsest mesh.
We observe convergence rates of similar order
in the $H^1$ and $L^\infty$ norms
even if the solution does not belong to $H^2(\Omega)$,
and better rates in the $L^2$ norm.

In general, our theory (\Cref{cor:convergence-FEM}) does not provide
a strategy on how to balance the regularization parameter $\varepsilon$
and the mesh resolution $h$.
In some of our examples, based on knowledge of the exact solution,
we were able to justify the observed behaviour with respect to
$\varepsilon$.

\bibliographystyle{amsplain}
\bibliography{references.bib}

\end{document}